\documentclass[reqno,11pt]{amsart}  
\usepackage[utf8]{inputenc}
\usepackage{graphicx}
\usepackage{subcaption} 
\usepackage{booktabs,tabularx,array}
\newcolumntype{Y}{>{\raggedright\arraybackslash}X}

\usepackage[colorinlistoftodos]{todonotes}
\usepackage{pgfplots} 
\usepackage{amsmath}
\usepackage{amssymb} 
\usepackage{mathtools}   
\usepackage{graphicx}
\usepackage{graphics}  
\usepackage{color}
\usepackage{verbatim} 
\usepackage{bbold}
\usepackage[normalem]{ulem} 
\usepackage[mathscr]{eucal}
\usepackage{upgreek}
\usepackage{enumerate} 
\usepackage{cite}
\usepackage{amsmath}
\usepackage{amsfonts}
\usepackage{amssymb}
\usepackage{bbm}
\usepackage{cancel}
\usepackage{graphicx}
\usepackage{multicol}
\usepackage[utf8]{inputenc}
\usepackage{framed}
\usepackage{setspace}

\usepackage{amsthm}
\usepackage{hyperref}
\usepackage{caption}
\usepackage{subcaption}
\usepackage{bbm}
\usepackage{graphicx}
\usepackage{tikz}
\usepackage[titletoc]{appendix}
\numberwithin{equation}{section}  

\usepackage[refpage,noprefix]{nomencl}
\usepackage{nomencl} 

\newcommand{\ssup}[1] {{\scriptscriptstyle{{#1}}}}
\newcommand{\gk}[1]{\left\{#1\right\}}

\newcommand\mycom[2]{\genfrac{}{}{0pt}{}{#1}{#2}}
\newcommand{\ek}[1]{\left[#1\right]}
\newcommand{\rk}[1]{\left(#1\right)}

\newcommand{\hk}[1]{^{\ssup{(#1)}}}
\newcommand{\abs}[1]{\left| #1 \right|}

\newcommand{\Lcal}   {{\mathcal L }}

\newcommand{\Ocal}   {{\mathcal O }}

\newcommand{\R}     {\mathbb{R}} 
 
\newcommand{\N}     {\mathbb{N}} 
\renewcommand{\P}   {\mathbb{P}} 
 
\newcommand{\E}     {\mathbb{E}}

 \newcommand{\ex}{{\rm e}} 
 \renewcommand{\d}{{\rm d}} 
 
\renewcommand{\L}{\Lambda}

\newcommand{\e}{\varepsilon}
\newcommand{\1}{\!\mathbbm{1}\!}

\renewcommand{\L}{\Lambda}

\newcommand{\Poi}{\mathrm{P}\tiny{\mathrm{oi}}}

\renewcommand{\P}{\mathbb{P}}

\newtheorem{remark}{Remark}[section]

\newtheorem{theorem}{Theorem}
\newtheorem{lemma}{Lemma}[section]
\newtheorem{definition}[lemma]{Definition}
\newtheorem{proposition}[lemma]{Proposition}

\newtheorem{cor}[lemma]{Corollary}

\newcommand{\Tr}{\mathrm{Tr}}

\usepackage[a4paper,margin=2cm]{geometry}

\renewcommand{\Poi}{\mathrm{Poi}}

\newcommand{\mass}{\mathrm{m}_L}

\newcommand{\Part}{\mathrm{Part}}

\newcommand{\can}{\mathrm{can}}

\newcommand{\dist}{\mathrm{dist}}

\newcommand{\T}{\mathbb{T}}
\newcommand{\pn}{\mathscr{N}}
\newcommand{\rhoc}{\rho_{\mathrm{c}}}
\newcommand{\ps}{\pn\hk{\mathrm{short}}}
\newcommand{\pmes}{\pn\hk{\mathrm{mes.}}}
\newcommand{\up}{\mathrm{up}}
\newcommand{\main}{\mathrm{main}}

\title{Fluctuations for the Critical Free Bose Gas}

\author{Quirin Vogel}
\address{Department of Statistics, University of Klagenfurt,
Klagenfurt, Austria}
\email{quirin.vogel@aau.at}


\begin{document}

\begin{abstract}
We study the critical free Bose gas from a probabilistic perspective with a focus on the three-dimensional case. We show that, unlike in the subcritical and supercritical regime, the critical asymptotics depend on the second heat-trace coefficient $a_1$ and hence on the geometry and boundary conditions. For $a_1<0$, long loops occur on the scale $L^2\log L$ and yield Poisson--Dirichlet statistics; for $a_1<0$, such loops are suppressed and the relevant fluctuations are produced by the short-loop bulk. We also derive asymptotics for the partition function, the one-particle reduced density matrix, and loop statistics. The fluctuation limits are governed either by a regularized Fredholm determinant of the Green operator or by Gaussian local central limit theorems, depending on the boundary regime.
\end{abstract}
\maketitle

\section{Introduction and Results}
\subsection{Introduction and model definition}
The Bose gas is a model from quantum statistical mechanics describing a collection of particles obeying Bose statistics, see \cite{a1981operator,lieb2005mathematics} for classical references. Probabilistic representations started with Feynman \cite{feynman1953atomic} and Matsubara \cite{matsubara1951quantum}. The mathematical foundations were later developed by Ginibre \cite{ginibre1971some}, while many important probabilistic results were subsequently obtained by Betz and Ueltschi and others; see, for example, \cite{betz2009spatial,betz2011spatialsmall}, as well as the earlier work \cite{suto1993percolation}. 

Neglecting the interaction between particles, the Bose gas can be probabilistically described as follows: let $\Lambda$ be the domain for the Bose gas, taken to be either the $d$-dimensional torus $\T^d$ or a compact region $\Lambda\subset\R^d$. Let $\Delta$ be the Laplace (Beltrami) operator on $\Lambda$, with Dirichlet, Neumann, Robin, or mixed Dirichlet--Robin boundary conditions in the case that $\Lambda$ has a boundary. Without loss of generality, assume that $\Lambda$ has unit volume, i.e., $\abs{\Lambda}=1$, and define the rescaled domain $\Lambda_L=L\Lambda$, where later $L\to\infty$. 

For $x\in\Lambda_L$ and $t>0$, let $\P_{x,x,t}^{(L)}$ denote the unnormalized Brownian bridge measure generated by $\Delta$ on $\Lambda_L$; see \cite{a1981operator,konig2025off} for a detailed construction. (Using $\Delta$ rather than $\Delta/2$ introduces an additional factor of $2$ in several later expressions but aligns with the mathematical physics literature.) For the inverse temperature $\beta>0$, define the loop measure
\begin{equation}\label{eq:BoseLoopMeasureDef}
    M_{L,\beta,N}
    =
    \sum_{j=1}^{N}\frac1j
    \int_{\Lambda_L}\d x\,
    \P_{x,x,\beta j}^{(L)}.
\end{equation}

The cut-off at $j=N$ arises naturally from the canonical ensemble; see Lemma~\ref{lem:cut-off}. Let $\P_{L,\beta,N}$ denote the Poisson point process with intensity measure $M_{L,\beta,N}$. A sample from $\P_{L,\beta,N}$ consists of an almost surely finite collection of loops $\{\omega_1,\ldots,\omega_n\}$. If a loop $\omega$ is sampled from $\P_{x,x,\beta j}^{(L)}$, we write $\ell(\omega)=j$ and interpret $\omega$ as representing $j$ particles. The total particle number is then
\begin{equation}
    \pn_L=\sum_{i=1}^n \ell(\omega_i).
\end{equation}
Note that $n$ is random and denotes the total number of loops in the sample.

Fixing a particle density $\rho>0$, the canonical free Bose gas is obtained by conditioning on the event $\pn_L=\rho L^d$: assume without loss of generality that $\rho L^d\in \N$ and set
\begin{equation}
    \P_{L,\beta,\rho}\hk{\can}(A)=\P_{L,\beta,\rho L^d}\rk{A \mid \pn_L=\rho L^d}\, .
\end{equation}
The corresponding partition function is
\begin{equation}
    Z_{L,\beta,\rho} =\P_{L,\beta,\rho L^d}\rk{
        \pn_L=\rho L^d}\, .
\end{equation}
The canonical free Bose gas undergoes a condensation transition as $\rho$ varies: there is a critical $\rhoc$ such that for $\rho>\rhoc$, macroscopic loops occur, see Figure~\ref{fig:three_panels_total} and Figure~\ref{fig:phasediag} for an illustration.

\begin{figure}[h]
     \centering
     \begin{subfigure}[b]{0.3\textwidth}
         \centering
         \includegraphics[width=\textwidth]{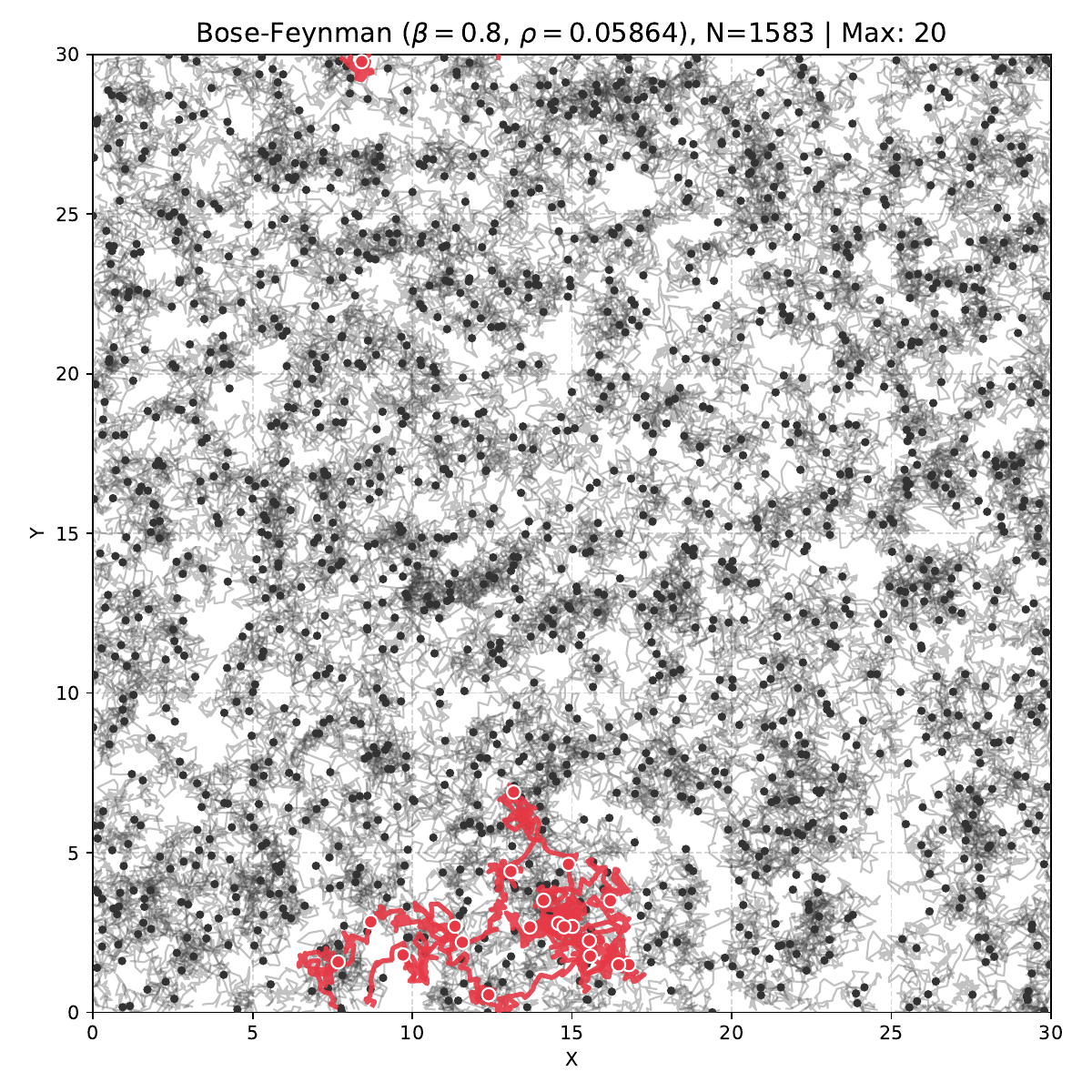}
         \caption{Below criticality}
         \label{fig:left}
     \end{subfigure}
     \hfill
     \begin{subfigure}[b]{0.3\textwidth}
         \centering
         \includegraphics[width=\textwidth]{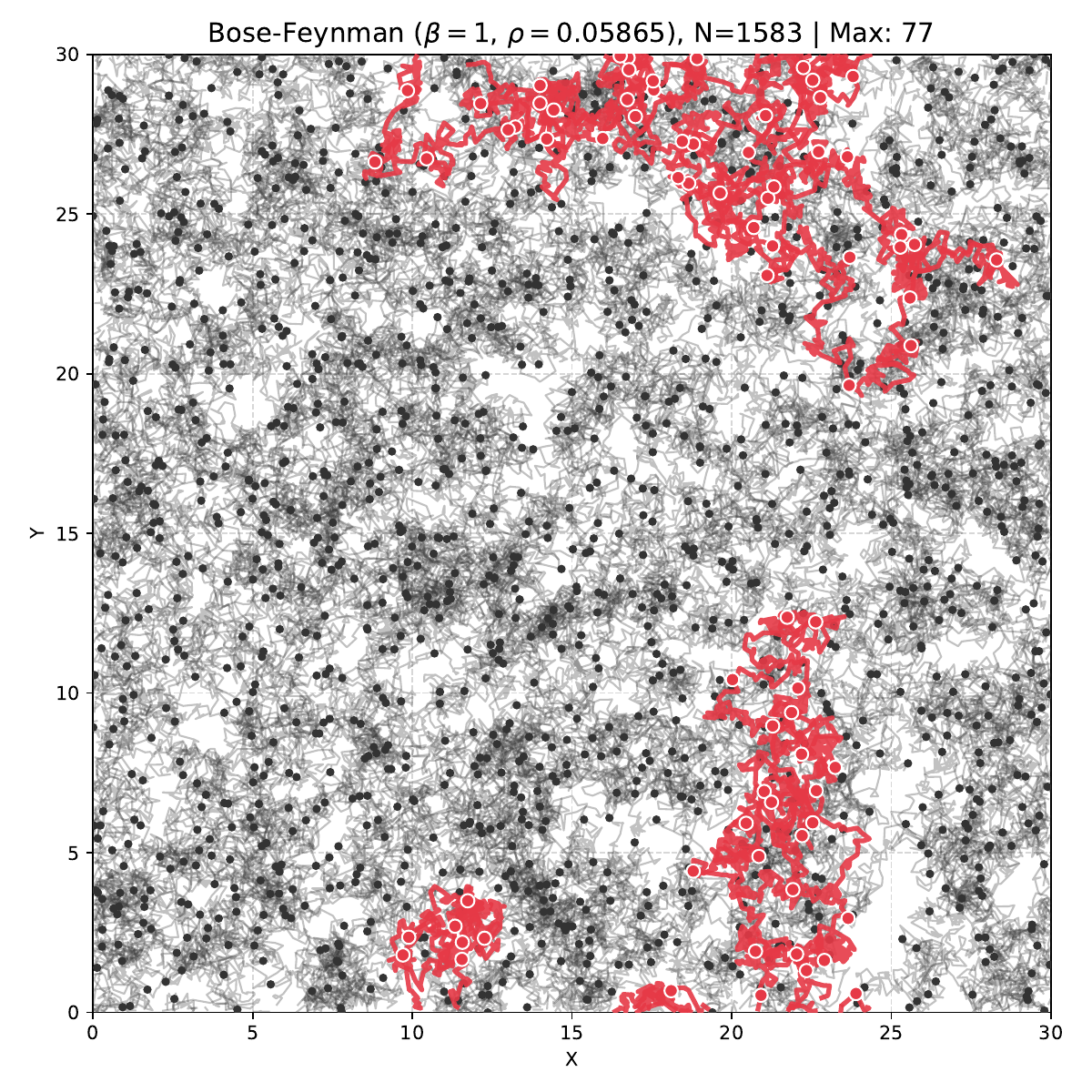}
         \caption{At approximate criticality}
         \label{fig:middle}
     \end{subfigure}
     \hfill
     \begin{subfigure}[b]{0.3\textwidth}
         \centering
         \includegraphics[width=\textwidth]{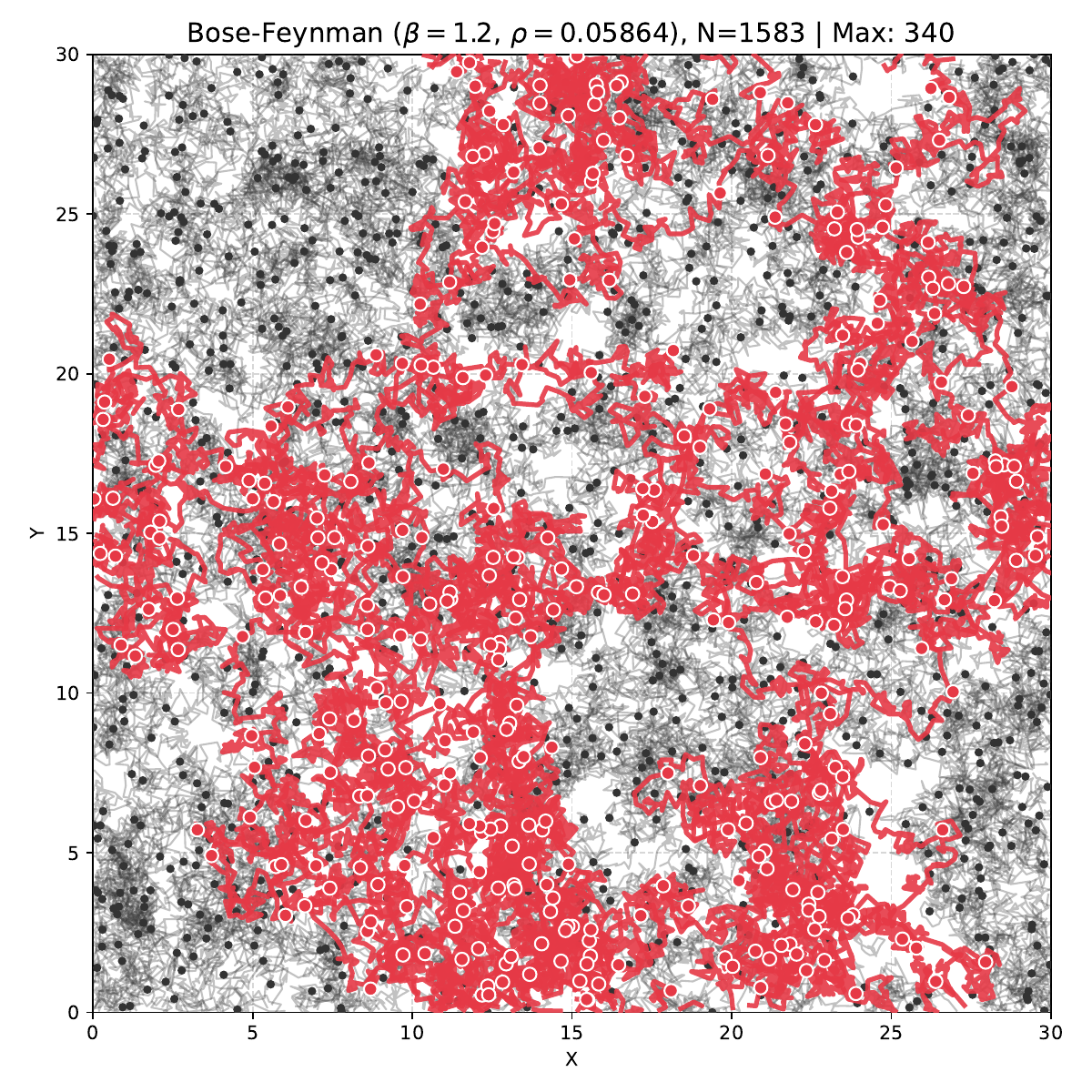}
         \caption{Above criticality}
         \label{fig:right}
     \end{subfigure}
     
     \caption{A simulation of the Bose gas for $\L_L=30\T^3$, projected onto the first two coordinates. The density is kept constant while $\beta$ is varied, to illustrate the nature of phase transition of the permutation. The largest loop is colored in red. The particles are represented by bold dots. }
     \label{fig:three_panels_total}
\end{figure}
\subsection{Past results}
In \cite{konig2025off}, a probabilistic analysis of $\P_{L,\beta,\rho}\hk{\can}$ and $Z_{L,\beta,\rho}$ was conducted. A phase transition occurs at the critical density
\begin{equation}
    \rhoc=(4\pi\beta)^{-d/2}\zeta(d/2)\, ,
\end{equation}
independent of the geometry of $\Lambda$ and the choice of boundary conditions. 
\begin{enumerate}
    \item for $\rho<\rhoc$, the partition function $Z_{L,\beta,\rho}$ decays at exponential speed, i.e., as $L\to\infty$ 
    \begin{equation}
        Z_{L,\beta,\rho}=\exp\rk{-f(\beta,\rho)L^d(1+o(1))}\quad\textnormal{with}\quad f(\beta,\rho)>0\, ,
    \end{equation}
    where the factor $f(\beta,\rho)$ is also called the \textit{free energy} of the system. Furthermore, $\P_{L,\beta,\rho}$ concentrates all the mass in finite loops: for any $\e>0$, there exists a $R_\e>0$ fixed such that
    \begin{equation}
         \lim_{L\to\infty}\P_{L,\beta,\rho}\hk{\can}\rk{\pn\hk{\le R_\e}>(\rho-\e)L^d}=1\, ,
    \end{equation}
    where $\pn\hk{\le R_\e}$ is the number of particles in loops of length less than or equal to $R_\e$
    \begin{equation}
        \pn\hk{\le R_\e}=\sum_{i=1}^n \ell(\omega_i)\1\gk{\ell(\omega_i)\le R_\e}\, .
    \end{equation}
    \item for $\rho>\rhoc$, the partition function $Z_{L,\beta,\rho}$ decays at stretched exponential speed, i.e., $Z_{L,\beta,\rho}=\ex^{-c_dL^{d-2}(1+o(1))}$. Furthermore, $\P_{L,\beta,\rho}\hk{\can}$ contains macroscopic loops: for any $0<\e<\rho-\rhoc$, there exists $R_\e>0$ fixed such that
    \begin{equation}
         \lim_{L\to\infty}\P_{L,\beta,\rho}\hk{\can}\rk{\pn\hk{\ge R_\e L^d}>(\rho-\rhoc-\e)L^d}=1\, ,
    \end{equation}
    where $\pn\hk{\ge R_\e L^d}$ is the number of particles in $R_\e$-macroscopic (i.e., containing a positive fraction of all particles) loops
    \begin{equation}
        \pn\hk{\ge R_\e L^d}=\sum_{i=1}^n \ell(\omega_i)\1\gk{\ell(\omega_i)\ge R_\e L^d}\, .
    \end{equation}
\end{enumerate}
In \cite{konig2025off}, the off-diagonal long range order (ODLRO) transition was also studied. Define the one-particle reduced density matrix $\gamma_{L,\beta,\rho}$ (a quantum analogue of the classical two-point correlation function)
\begin{equation}
     \gamma_{L,\beta,\rho}(x,y)=\sum_{r=1}^{\rho L^d}\frac{\P_{x,y,\beta r}\hk{L}(1)\P_{L,\beta,\rho L^d}\rk{\pn_L=\rho L^d-r}}{\P_{L,\beta,\rho L^d}\rk{\pn_L=\rho L^d}}\, ,
\end{equation}
where $\P_{x,y,\beta r}\hk{L}$ is the unnormalized bridge measure from $x$ to $y$ on $\L_L$. In \cite{konig2025off} it was shown that $\gamma_{L,\beta,\rho}$ undergoes the ODLRO phase transition: for $x,y\in \L_L$ sufficiently far apart, the first order term\footnote{We write $A\sim B$ if $A=B(1+o(1))$} of $\gamma_{L,\beta,\rho}(x,y)$ is given by
\begin{equation}
    \gamma_{L,\beta,\rho}(x,y)\sim\begin{cases}
        \ex^{-\Ocal(\abs{x-y})}&\textnormal{ if }\rho<\rhoc\, ,\\
        (\rho-\rhoc)\phi\hk{1}_L(x)\phi\hk{1}_L(y)&\textnormal{ if }\rho>\rhoc\, ,
    \end{cases}
\end{equation}
where $(\phi\hk{1}_L)^2$ is the density of the invariant distribution for the $h$-transformed process on $\L$, scaled by $L$. For periodic or Neumann boundary conditions, $\phi\hk{1}_L\equiv 1$, illustrating the off diagonal long-range order in the system.

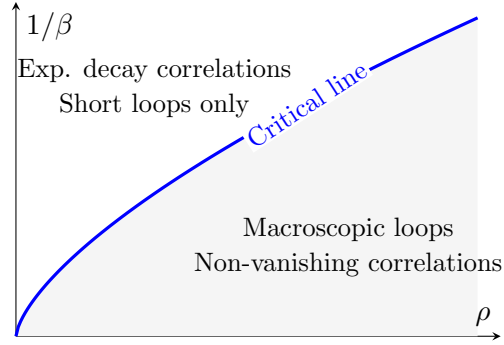
\begin{figure}[ht]
\centering

\begin{tikzpicture}

\begin{axis}[
    axis lines=middle,
    xlabel={$\rho$},
    ylabel={$1/\beta$},
    xmin=0, xmax=1.05,
    ymin=0, ymax=1.05,
    xtick=\empty,
    ytick=\empty,
    width=8cm,
    height=6cm,
    domain=0:1,
    samples=200,
    clip=false,
    enlargelimits=false,
]

\addplot[
    domain=0:1,
    draw=none,
    fill=gray!8
]
{x^(2/3)}
\closedcycle;

\addplot[
    very thick,
    blue
]
{x^(2/3)};

\node[align=center] at (axis cs:0.3,0.78)
{
\small Exp. decay correlations\\
\small Short loops only
};

\node[align=center] at (axis cs:0.72,0.28)
{
\small Macroscopic loops\\
\small Non-vanishing correlations
};

\node[
    blue,
    rotate=33,
    fill=white,
    inner sep=1pt
] at (axis cs:0.63,0.73)
{
\small Critical line
};

\end{axis}

\end{tikzpicture}

\caption{
Phase diagram of the free Bose gas in dimension $d=3$.
The present article studies the behavior on the critical line.
}\label{fig:phasediag}

\end{figure}

\subsection{High-level summary}
In this article, the case $\rho=\rhoc$ is studied. The critical regime exhibits substantially richer behavior. 
In contrast to the subcritical and supercritical phases, where the relevant quantities display exponential or asymptotically constant behavior, criticality is characterized by polynomial and logarithmic corrections, especially in dimension three. We derive asymptotics for the partition function, the one-particle reduced density matrix and loop statistics, as previously done in the subcritical and supercritical phases.

To state the results, we need to introduce the \textit{first boundary heat coefficient} $a_1$, which is a signed multiple of the boundary measure of $\L$:
\begin{equation}
    a_1=\begin{cases}
        -\frac{\abs{\partial\L}}{4(4\pi)^{(d-1)/2}}&\textnormal{ if }\Lambda\subset\R^d\textnormal{ and Dirichlet boundary conditions}\, ,\\
        +\frac{\abs{\partial\L}}{4(4\pi)^{(d-1)/2}}&\textnormal{ if }\Lambda\subset\R^d\textnormal{ and Neumann/Robin boundary conditions}\, ,\\
        0&\textnormal{ if }\Lambda=\T^d\, .
    \end{cases}
\end{equation}
See Section~\ref{sec:geoandBC} for an in-depth introduction to boundary conditions and the first boundary heat coefficient $a_1$.

The behavior of the critical Bose gas depends on the sign and magnitude of $a_1$. Denote by $\lambda_i$ the $i$-th eigenvalue of $-\Delta\ge 0$ on $\Lambda$. 

The main results for our paper are for three dimensions, the physically most relevant case.

Before stating the results as theorems, we summarize the behavior of the Feynman representation of the critical free Bose gas in three dimensions.

\noindent

\begin{table}[t]
\centering
\caption{Informal trichotomy for the critical free Bose gas in dimension \(d=3\). $1+o(1)$ terms omitted for legibility.}
\label{tab:d3-trichotomy}
\footnotesize
\setlength{\tabcolsep}{3pt}
\renewcommand{\arraystretch}{1.18}

\begin{tabularx}{\textwidth}{@{}p{0.15\textwidth}YYY@{}}
\toprule
&
\begin{tabular}[c]{@{}c@{}}
\textbf{Boundary deficit}\\
\(\mathbf{a_1<0}\)
\end{tabular}
&
\begin{tabular}[c]{@{}c@{}}
\textbf{Marginal case}\\
\(\mathbf{a_1=0}\)
\end{tabular}
&
\begin{tabular}[c]{@{}c@{}}
\textbf{Boundary excess}\\
\(\mathbf{a_1>0}\)
\end{tabular}
\\
\midrule

Typical domain
&
Dirichlet-type boundary.
&
Torus or closed manifold.
&
Neumann/Robin-type boundary.
\\
\hline
Mechanism
&
Short loops fall short of \(\rho_{\mathrm c}L^3\) by a surface-order term. Long loops supply the missing mass.
&
The \(a_1L^2\log L\) correction vanishes. The mixing scale itself becomes visible.
&
Short loops exceed \(\rho_{\mathrm c}L^3\). A negative chemical potential suppresses the excess.
\\
\hline
Partition function
&
\(Z_{L,\beta,\rho_{\mathrm c}}
= L^{-2+2\lambda_1a_1-\1\gk{\{\lambda_1=0\}}}\).
&
\(Z_{L,\beta,\rho_{\mathrm c}}\asymp L^{-3}\).
&
\(Z_{L,\beta,\rho_{\mathrm c}}
= \exp\{-128\pi^2a_1^3\log^3 (L)/3\}\).
\\
\hline
Loop picture
&
Long loops carry
\(m_L=cL^2\log L\) particles and have \(\operatorname{PD}(1)\) statistics.
&
Borderline regime. Loops of order \(L^2\) contribute; there is no logarithmic long-loop mass.
&
The natural tilted cutoff is \(|\mu|^{-1}\asymp L^2/\log^2 L\). Long loops are suppressed.
\\
\hline
ODLRO
&

\(\gamma_L(x,y)\sim cL^{-1}\log L\).
&
\(\gamma_L(x,y)\asymp L^{-1}\).
&

\(\gamma_L(x,y)=L^{-1-|x-y|\pi a_1/2L}\).
\\
\hline
Fluctuations
&
Scale \(\beta^{-1}L^2\). Limit described by a regularized Fredholm determinant.
&
Scale \(L^2\). Borderline local limit governed by mixing-scale loops.
&
Tilted Gaussian local limit with inverse scale
\(\nu=\sqrt{\log L}/L^2\), i.e. fluctuations of size \(L^2/\sqrt{\log L}\).
\\
\hline
Interpretation
&
Criticality resembles the supercritical phase: long loops repair a boundary-induced deficit.
&
The marginal case is genuinely borderline.
&
Criticality is enforced by a downward tilt of the short-loop cloud.
\\

\bottomrule
\end{tabularx}
\end{table}
Dimensions four and higher are characterized by Gaussian fluctuations for all boundary conditions, though the trichotomy induced by the first heat coefficient remains.
\subsection{Statement of results for three dimensions}
Let us now state the main results more carefully. Recall that the mixing scale for a Brownian motion on $\L_L$ is $L^2$.
\begin{theorem}\label{thm:bothneg}\textbf{Negative first boundary heat coefficient} $\boldsymbol{a_1<0}$. 

\noindent
    Let $d=3$ and fix $\rho=\rhoc=(4\pi\beta)^{-3/2}\zeta(3/2)$.
    \begin{enumerate}
    \item As $L\to\infty$, the partition function satisfies     \begin{equation}\label{eq:threeDimePartFuncStatement}
            Z_{L,\beta,\rhoc}=\P_{L,\beta,\rhoc L^3}\rk{\pn_L=\rhoc L^3}=L^{-2+2\lambda_1a_1-\1\gk{\lambda_1=0}+o(1)}
    \end{equation}
    \item The density of particles in loops of length exceeding $L^2$ is asymptotically $-2a_1L^{-1}\log(L)$. \textnormal{Formally}: for $\alpha=\log\log(L)/\log\log\log(L)$ and any $\e>0$
    \begin{equation}\label{eq:themoprdm}
        \lim_{L\to\infty}\P_{L,\beta,\rhoc}\hk{\can}\rk{\pn\hk{\ge \alpha L^2}\in -2a_1L^{2}\beta^{-1}\log(L)(1-\e,1+\e)}=1\, .
    \end{equation}
    \item Write $\phi\hk{1}=\phi\hk{1}_1$ for the invariant distribution on $\Lambda$. The one-particle reduced density matrix satisfies
    \begin{equation}
        \gamma_{L,\beta,\rho}(x,y)\sim -2a_1\beta^{-1}\phi\hk{1}(x/L)\phi\hk{1}(y/L) L^{-1}\log(L)\, ,
    \end{equation}
    if there exists $\e>0$ such that $\dist(x,y)>\e L$ and the distance to the boundary of both points is also lower bounded by $\e L$.
    \item The long loops satisfy a Poisson--Dirichlet scaling limit: denote $\Lcal_i$ the number of particles contained in the $i$-th largest loop in the sample. Under the measure $\P_{L,\beta,\rhoc}\hk{\can}$, for $m_L=-2a_1\beta^{-1}L^2\log(L)$, the vector
    \begin{equation}
        \rk{\frac{\Lcal_1}{m_L},\frac{\Lcal_2}{m_L},\ldots,\frac{\Lcal_n}{m_L}}\, ,
    \end{equation}
    converges weakly in the product topology of $\ell^\infty$ to a Poisson--Dirichlet distribution with parameter 1.
    \end{enumerate}
\end{theorem}
If the elasticity constant is positive, the behavior of the system is completely different.
\begin{theorem}\label{thm:posboundnegla}\textbf{Positive first boundary heat coefficient} $\boldsymbol{a_1>0}$. 

\noindent
    Let $d=3$ and fix $\rho=\rhoc=(4\pi\beta)^{-3/2}\zeta(3/2)$.
    \begin{enumerate}
    \item As $L\to\infty$, the partition function satisfies     \begin{equation}\label{eq:threeDimePartFuncStatement1}
            Z_{L,\beta,\rhoc}=\P_{L,\beta,\rhoc L^3}\rk{\pn_L=\rhoc L^3}=\exp\rk{-128\pi^2a_1^3\log^3(L)/3(1+o(1))}\, .
    \end{equation}
    \item With high probability, there are no loops of length greater than $L^2/\log L$. Formally: there exists $c>0$ such that
    \begin{equation}
        \lim_{L\to\infty}\P_{L,\beta,\rhoc}\hk{\can}\rk{\pn\hk{\ge c L^2\log^{-1}(L)}=0}=1\, .
    \end{equation}
    The equivalent \textnormal{chemical potential} is of scale $\log^2L/L^2$, and therefore for a compact box, no loops of length exceeding that scale are present.
    \item Assume that $\abs{x-y}=\alpha L$ with $\alpha$ fixed. The one-particle reduced density matrix then satisfies
    \begin{equation}\label{eq:040520261}
        \gamma_{L,\beta,\rho}(x,y)=L^{-1-\alpha \pi a_1/2+o(1)}\, .
    \end{equation}
    \item The bulk satisfies a \textnormal{local} central limit theorem at scale
    \begin{equation}
        \nu^{-1}=\frac{L^2}{\sqrt{\log L}}\, ,
    \end{equation}
    with Gaussian limit law, see Lemma~\ref{lem:localCLT} for a precise statement.
    \end{enumerate}
\end{theorem}
The zero first boundary heat coefficient case can be seen as the limit $\lim_{a_1\uparrow 0^-}$:
\begin{theorem}\textbf{Zero first boundary heat coefficient} $\boldsymbol{a_1=0}$.

\noindent
    Let $d=3$ and fix $\rho=\rhoc=(4\pi\beta)^{-3/2}\zeta(3/2)$.
    \begin{enumerate}
        \item There exists $C>0$ such that for $L\ge 1$
        \begin{equation}
            C^{-1}L^{-3}\le \P_{L,\beta,\rhoc L^3}\rk{\pn_L=\rhoc L^3}\le CL^{-3}\, .
        \end{equation}
        \item  If there exists $\e>0$ such that $\dist(x,y)>\e L$ and the distance to the boundary of both points is also lower bounded by $\e L$, there exists $C>0$ such that for $L\ge 1$
        \begin{equation}
            C^{-1}L^{-1}\le \gamma_{L,\beta,\rho}(x,y)\le CL^{-1}\, .
        \end{equation}
    \end{enumerate}
\end{theorem}
If $\lambda_1>0$, then the total particle number satisfies a non-Gaussian fluctuation limit.
\begin{proposition}\label{prop:cltMain}
    If $d=3$, $\rho=\rhoc$, and $\lambda_1>0$, then $\pn_L$ satisfies
    \begin{equation}
        \frac{\pn_L-\E[\pn_L]}{\beta^{-1}L^2}\xrightarrow[L\to\infty]{\mathrm{(d)}}\chi_\zeta\, .
    \end{equation}
   where $\chi_\zeta$ has characteristic function $\psi$ given by the second regularized Fredholm determinant (see \cite[Chapter 9]{simon2005trace} for a reference on Fredholm determinants) of $\mathrm{Id}-i\Delta^{-1}$
    \begin{equation}
        \psi(t)=\det\nolimits_2\rk{\mathrm{Id}-it\Delta^{-1}}^{-1}\, .
    \end{equation}
    Furthermore, $\chi_\zeta$ can also be represented as an infinite sum of independent recentered Gamma random-variables
    \begin{equation}\label{eq:repreGammaFun}
        \chi_\zeta=\sum_{k\ge 1}\rk{\mathrm{Gamma}(1,\lambda_k)-\frac{1}{\lambda_k}}\, .
    \end{equation}
\end{proposition}
\subsection{Statement of results for higher dimensions}
In higher dimensions, the scales do not overlap and hence we see stretched-exponential decay, rather than the polynomial decay for the three-dimensional case.
\begin{theorem}
Assume $d\ge 4$ throughout and let $\rhoc=(4\pi\beta)^{-d/2}\zeta(d/2)$
    \begin{enumerate}
        \item For $a_1<0$ the partition function satisfies
        \begin{equation}
            \P_{L,\beta,\rhoc L^d}\rk{\pn_L=\rhoc L^d}=\exp\rk{a_1\lambda_1\beta^{(3-d)/2}\zeta(d/2-1/2)L^{d-3}(1+o(1))}\, .
        \end{equation}
        \item For $a_1>0$ the partition function satisfies
        \begin{equation}
            \P_{L,\beta,\rhoc}\rk{\pn_L=\rhoc L^d}=\exp\rk{-f(d)L^{d-2}\log^{-\1\gk{d=4}}(L)(1+o(1))}\, ,
        \end{equation}
        with
        \begin{equation}
            f(d)=a_1^2(4\pi)^{d/2}\beta^{1-d/2}\frac{\zeta(d/2-1/2)-1}{C_d}\, ,
        \end{equation}
        with
        \begin{equation}
            C_d=\begin{cases}
                \zeta(d/2-1)&\textnormal{ if }d\ge 5\, ,\\
                1&\textnormal{ if }d=4\, .
            \end{cases}
        \end{equation}
        \item For $a_1<0$, the one-particle reduced density matrix decays at linear speed
        \begin{equation}
            \gamma_{L,\beta,\rhoc}(x,y)\sim -\frac{a_1\zeta(d/2-1/2)}{\beta^{(d-1)/2}L}\phi\hk{1}(x/L)\phi\hk{1}(y/L)\, .
        \end{equation}
        \item For $a_1>0$, the one-particle reduced density matrix decays at stretched-exponential speed: for $\abs{x-y}^2=\alpha^2L^2$
        \begin{equation}
            \gamma_{L,\beta,\rhoc}(x,y)=\exp\rk{-\sqrt{\frac{c_d\alpha^2 L}{1+\log(L)\1\gk{d=4}}}(1+o(1))}\, ,
        \end{equation}
        where the constant $c_d$ is explicit and is given in \eqref{eq:040520263} (for $d\ge 5$) and in \eqref{eq:040520264} for $d=4$ respectively.
    \end{enumerate}
\end{theorem}
\begin{remark}
    \noindent
    For $a_1<0$ note the similarity to the supercritical case in \cite{konig2025off,bai2025proof}: there, the partition function decays stretched-exponentially with exponent $L^{d-2}$ (Brownian scaling). At criticality, our results show that the exponent reduces by one, for supercritical dimensions $d\ge 4$.
    
    For $a_1>0$ and $d\ge 5$, the decay matches that of $\rho>\rhoc$ from \cite{konig2025off}. However, for $\rho>\rhoc$ the partition captures a \textnormal{upward} large deviation event of \textnormal{volume} order (exceeding the expectation by $(\rho-\rhoc)L^d$) whereas for $\rho=\rhoc$ and $a_1>0$ the partition function is an \textnormal{downward} large deviation event of \textnormal{surface} order. We see the matching of exponents as a curious coincidence.
\end{remark}
Finally, we state results for $\Lambda=\T^d$ separately:
\begin{theorem}\label{thm:torusHighD}
    Let $\Lambda=\T^d$, $d\ge4$. Then, there exists $C>0$ such that
    \begin{equation}
        C^{-1}L^{-d}\le \P_{L,\beta,\rhoc  L^d}\rk{\pn_L=\rhoc L^d}\le CL^{-d}\, .
    \end{equation}
    Furthermore
    \begin{equation}
    \gamma_{L,\beta,\rhoc}(x,y)\asymp 
    \begin{cases}
        L^{-2}\sqrt{\log(L)}, & d=4,\\
        L^{-d/2}, & d\ge5.
    \end{cases}
\end{equation}
\end{theorem}
\subsection{Literature remarks} Key classical references for a probabilistic analysis of the Bose gas are \cite{ginibre1971some,betz2009spatial,betz2011spatialsmall}. Recent works include the aforementioned \cite{konig2025off}, where the subcritical and supercritical regimes were studied. In \cite{bai2025proof}, this was generalized to the free Bose gas with a trap. Both of these works were inspired by \cite{Vogel2023Emergence,Dickson2024} where the appearance of random interlacements was shown in the thermodynamic limit, although for free boundary conditions. A probabilistic representation of the supercritical free Bose gas via Bosonic random interlacements was given in \cite{armendariz2019gaussian}. 

We also mention the recent work \cite{drewitz2023critical,cai2024one, drewitz2025critical,cai2025one,duminil2025new} that analyzes the critical regime in more correlated models induced by the Gaussian free field. Note that as $\beta\to 0$, $M_{L,\beta,N}$ converges to the Le Jan loop measure $\int_0^\infty\d t\int_\L \tfrac{\P_{x,x,t}\hk{\L}}{t}\d x$, see \cite{adams2020space} for a reference on the lattice and \cite{le2010markov,le2024random} for a reference on the Le Jan loop measure and its connection to the Gaussian free field.
\subsection{Contributions and comparison with previous work}

Previous probabilistic analyses of the Feynman--Kac representation have described the free Bose gas below and above the critical density, including off-diagonal long-range order, long-loop
statistics, and the emergence of random interlacements; see \cite{betz2011spatial,Vogel2023Emergence,Dickson2024,
konig2025off,bai2025proof} for example. This work considers the canonical ensemble exactly at the critical density $\rho=\rhoc$. To the best of our knowledge, the boundary-sensitive finite-volume behavior at criticality has not previously been resolved probabilistically at the simultaneous level of the canonical normalization, the one-particle reduced density matrix, and the loop-length process.

Our first contribution is a critical trichotomy governed by the first boundary coefficient $a_1$ in the heat-trace expansion. Although this coefficient is of one order lower order than the volume term, in
dimension $d=3$ its accumulated contribution is of order $L^2\log L$, and determines how the canonical constraint at $\rho=\rhoc$ is realized. When $a_1<0$, the short-loop soup has a boundary-induced deficit of order $L^2\log L$. This deficit is supplied by long loops, whose normalized lengths converge to a Poisson--Dirichlet distribution, and it produces an $ L^{-1}\log L$ contribution to the one-particle reduced density matrix. When $a_1>0$, the short-loop
soup instead has an excess. The canonical constraint is then realized through a negative chemical potential of order $L^{-2}\log^2 L$, which suppresses long loops and leads to different asymptotics for both the canonical normalization and the off-diagonal decay. The case $ a_1=0$ is genuinely marginal and exposes the Brownian mixing scale $L^2$.

Our second contribution is a fluctuation theory that distinguishes these mechanisms. In dimension $d=3$, when $\lambda_1>0$, the centered particle number has a non-Gaussian infinitely divisible limit on the $L^2$-scale. Its characteristic function is expressed through a regularized Fredholm determinant of the Green
operator (or as an infinite sum of centered Gamma random variables); see Proposition~\ref{prop:cltMain}. By contrast, in the regime $a_1>0$, the boundary-induced tilt leads to a Gaussian central limit theorem on the smaller scale $L^2/\sqrt{\log L}$; see Lemma~\ref{lem:localCLT}. In higher dimensions the relevant fluctuations are Gaussian, with $d=4$ appearing as the critical dimension through an additional logarithmic correction.

The Minakshisundaram--Pleijel/heat-trace expansion used here is classical. The new point is its  probabilistic role in the canonically conditioned loop soup: the first boundary term determines whether criticality is achieved by creating long loops or by suppressing the short-loop cloud. The proof combines the short-time heat-trace expansion with the large-time spectral approximation of the heat kernel.

A principal technical difficulty is that the relevant long-loop scale at criticality is already $L^2\log L$. The cutoff of order $L^2\log^2 L$ used in earlier noncritical analyses is therefore
not available. We introduce instead the intermediate scale
\begin{equation}
        r_L=\frac{L^2\log\log L}{\log\log\log L}
\end{equation}
which lies above the Brownian mixing scale but below the critical long-loop scale. This separation permits simultaneous control of the short-loop surface correction, the long-loop ground-state contribution, and the canonical conditioning.
\subsection{Open Questions}
There are several unresolved questions concerning the probabilistic analysis of the critical free Bose gas. For example, can one identify the lower order terms in both free energy and the partition function? We believe that a refinement of our method should make that possible, by separating loops at scale $L^2$ and incorporating the full spectrum into the law of long loops, while also proving a local CLT for short loops.

Furthermore, the extension of our results to confining potentials also remains open. Here, the expansion of the Schr\"odinger operator trace at short times is much more nuanced and strongly depends on the potential, not only the dimension. We refer the reader to \cite{cognola2006heat} for steps in that direction.

Finally, Theorem~\ref{thm:bothneg} suggests that sample $\P_{L,\beta,\rhoc}$ can be described by sampling a loop soup with short loops and sprinkling tilted random interlacements on top. Steps in this direction are undertaken in \cite{bouchot2024confined}, see also \cite{Vogel2023Emergence} for the free boundary case.
\section{Preparatory results}
\subsection{The Minakshisundaram--Pleijel expansion}\label{sec:geoandBC}
This short section recaps the asymptotics for the trace of the heat kernel, also known as Minakshisundaram--Pleijel expansion, see \cite{minakshisundaram1949some} for the original paper. While our theorems are stated for $\Lambda\subset\R^d$ or $\Lambda=\T^d$, our proofs rely only on the properties of the Brownian motion on a Riemannian manifold encoded by the Laplace--Beltrami operator.

For a $d$-dimensional Riemannian manifold $\Lambda$ with metric tensor $g=\rk{g_{ij}}$ we analyze the asymptotics of
\begin{equation}
    Z(t)=Z_\L(t)=\Tr\rk{\ex^{-t\Delta}}\quad\textnormal{as}\quad t\to 0\, ,
\end{equation}
where $\Delta$ is the Laplace--Beltrami operator on $\Lambda$
\begin{equation}
    \Delta=\frac{1}{\sqrt{\det g}}\frac{\partial }{\partial x_i}g^{ij}\sqrt{\det g}\frac{\partial }{\partial x_j}\, ,
\end{equation}
subject to the boundary conditions in the case where $\Lambda$ is open. In \cite{mckean1967curvature}, it was shown that as $t\to 0$
\begin{equation}\label{eq:MPexpansion}
    Z(t)=a_0t^{-d/2}+a_1t^{-(d-1)/2}+a_2t^{-(d-2)/2}+\Ocal(t^{-(d-3)/2})\, ,
\end{equation}
with
\begin{equation}
    \begin{split}
        a_0&=(4\pi)^{-d/2}\abs{\L}\, ,\\
        a_1&=\pm \frac{\abs{\partial\L}}{4(4\pi)^{(d-1)/2}}\1\gk{\L\textnormal{ open}}\, ,
    \end{split}
\end{equation}
where $\abs{\L}$ the (Riemannian) volume of $\Lambda$ and $\abs{\partial\L}$ denotes the (Riemannian) surface area of the boundary of $\Lambda$. Furthermore, $a_1>0$ if $\Delta$ satisfies Neumann boundary conditions and $a_1<0$ for Dirichlet boundary conditions. An explicit formula for $a_2$ was also given as well as higher order terms for closed manifolds, see \cite[Eq. 5a]{mckean1967curvature}. In \cite[Eq. (5.35)]{vassilevich2003heat} it was shown that the above formula for $a_1$ continues to hold in the case of Robin boundary conditions
\begin{equation}
    \rk{\nabla_N \phi+S\phi}\big|_{\partial \Lambda}=0\, ,
\end{equation}
where $\nabla_N $ is the covariant derivative, i.e., the generalization of the normal vector for smooth subsets of $\R^d$ and $S\colon \partial \L\to \R$ is the Robin Coefficient ($S=\infty$ represents Dirichlet and $S=0$ represents pure Neumann boundary conditions). Finally, an explicit formula for $a_2$ is given by the curvature integral for case that $\L$ is closed (i.e. $a_1=0$)
\begin{equation}
    a_2=\frac{1}{3(4\pi)^{d/2}}\int_\L K\, ,
\end{equation}
where $K$ is the scalar curvature of $\L$ (the
negative of the spur $\sum_{i<j}R^{ij}_{ij}$ of the Ricci tensor). Importantly, for $\L=\T^d$, $K\equiv 0$ and hence $a_2=0$. In fact, for the torus, all higher order terms vanish $a_i=0$ for $i\ge 1$, see \cite[Lemma A.2]{benfatto2005limit} for an explicit expression of $Z(t)$ for the torus from which this claim follows. For other manifolds, $a_2$ can be positive or negative, depending on the manifold. This influences the critical Bose gas in dimensions $d\ge 4$.

We refer to \cite{vassilevich2003heat,gilkey2018invariance} for further properties $Z(t)$.

The Minakshisundaram--Pleijel zeta function is defined as
    \begin{equation}
        K(s)=\frac{1}{\Gamma(s)}\int_0^\infty t^{s-1}Z(t)\d t\, .
    \end{equation}
    Due to its ubiquity in our calculations, we abbreviate
    \begin{equation}\label{eq:zetafu}
        \zeta_\L(t)=t^{-1}Z(t)\, .
    \end{equation}
\subsection{Poisson point process representation}
We state the following standard results from spectral theory, see \cite{port2012brownian} for a reference.
\begin{definition}
    Denote by $p_t\hk{L}(x,y)$ the heat-kernel of Brownian motion $\L_L$ from $x$ to $y$ in time $t$. It has a spectral representation (by Brownian scaling)
    \begin{equation}\label{eq:spectraldec}
        p_t\hk{L}(x,y)=L^{-d}\sum_{k\ge 1}\ex^{-\lambda_k tL^{-2}}\phi\hk{k}(x/L)\phi\hk{k}(y/L)\, ,
    \end{equation}
    where $\lambda_k$ is the $k$-th eigenvalue of $-\Delta$ on $\L$ and $\phi\hk{k}$ is the associated $L^2$ normalized $k$-th eigenfunction (where we furthermore choose $\phi\hk{1}>0$).

    For $\beta>0$, $L>1$ and $j\in \N$ set
    \begin{equation}
        t_{j,L}=t_j=\int_{\L_L}\d x p_{\beta j}\hk{L}(x,x)\, .
    \end{equation}
\end{definition}
By the spectral representation
\begin{equation}\label{eq:specTjweights}
    t_j=\sum_{k\ge 1}\ex^{-\lambda_k\beta jL^{-2}}=Z(\beta jL^{-2})\, .
\end{equation}
Thus, combining \eqref{eq:MPexpansion} with \eqref{eq:spectraldec} and \eqref{eq:specTjweights} yields
\begin{cor}\label{cor:weightExp}
As $L\to\infty$, the following holds:
    \begin{enumerate}
        \item For $j=o(L^2)$
        \begin{equation}\label{eq:tjForSmallj}
            t_j=\frac{L^d}{(4\pi\beta j)^{d/2}}+a_1\frac{L^{d-1}}{(\beta j)^{(d-1)/2}}+a_2\frac{L^{d-2}}{(\beta j)^{(d-2)/2}}+\Ocal\rk{L^{d-3}j^{-(d-3)/2}}\, .
        \end{equation}
        \item For $j\gg L^2$
        \begin{equation}\label{eq:firstEvDom}
            t_j=\ex^{-\lambda_1\beta jL^{-2}}\rk{1+\Ocal\rk{\ex^{-\beta jL^{-2}(\lambda_2-\lambda_1)}}}\, .
        \end{equation}
    \end{enumerate}
\end{cor}
We also present a brief justification of the cut-off from \eqref{eq:BoseLoopMeasureDef}:
\begin{lemma}\label{lem:cut-off}
    Set $M_{L,\beta,\mu}=\sum_{j\ge 1}\frac{\ex^{\beta \mu j}}{j}\int_{\L_L}\d x\P_{x,x,\beta j}\hk{L}$ and write $\P_{L,\beta,\mu}$ for the associated Poisson point process. Then, if $\lambda_1=0$, for $d\ge 3$ and any $\rho>0$
    \begin{equation}
        \lim_{\mu\uparrow 0}\ex^{ \sum_{j> \rho L^d}\frac{\ex^{\beta \mu j}}{j}t_j}\P_{L,\beta,\mu}\rk{\pn_L=\rho L^d}=\P_{L,\beta,\rhoc L^d}\rk{\pn_L=\rho L^d}\, .
    \end{equation}
    Furthermore, as $L\to\infty$, uniformly in $\mu<0$
    \begin{equation}
        \ex^{ \sum_{j> \rho L^d}\frac{\ex^{\beta \mu j}}{j}t_j}=\frac{\ex^{-\tilde{\gamma}}}{-\beta\mu \rho L^d}(1+o(1))\, ,
    \end{equation}
    where $\tilde{\gamma}\approx 0.5772156649$ is the Euler--Mascheroni constant.
\end{lemma}
\begin{proof}
    Denote
    \begin{equation}
        p\hk{\le \rho}(\mu)=\sum_{j=1}^{\rho L^d}\frac{\ex^{\beta \mu j}}{j}t_j\quad\textnormal{and}\quad p\hk{>\rho}(\mu)=\sum_{j> \rho L^d}\frac{\ex^{\beta \mu j}}{j}t_j\, .
    \end{equation}
    The Poisson point process construction directly implies that
    \begin{equation}
    \begin{split}
        \P_{L,\beta,\mu}\rk{\pn_L=\rho L^d}&=\ex^{-(p\hk{\le \rho}(\mu)+p\hk{>\rho}(\mu))}\sum_{n\ge 1}\sum_{\mycom{j_1,\ldots,j_n\ge 1}{j_1+\ldots+j_n=\rho L^d}}\frac{\ex^{\beta \mu \rho L^d}}{j_1\cdots j_n}\prod_{r=1}^n t_{j_r}\\
        &= \ex^{-(p\hk{\le \rho}(\mu)-p\hk{\le \rho}(0)+p\hk{>\rho}(\mu))}\P_{L,\beta,\rho L^d}\rk{\pn_L=\rho L^d}\, .
        \end{split}
    \end{equation}
    As $L$ remains fixed
    \begin{equation}
        \lim_{\mu\uparrow 0}p\hk{\le \rho}(\mu)-p\hk{\le \rho}(0)=0\, ,
    \end{equation}
    and hence it remains to calculate $p\hk{>\rho}(\mu)$. By \eqref{eq:firstEvDom} and the fact that $d>2$
    \begin{equation}
        \sum_{j> \rho L^d}\frac{\ex^{\beta \mu j}}{j}t_j=o(1)+\sum_{j> \rho L^d}\frac{\ex^{\beta \mu j}}{j}=-\log(-\beta \mu)-\log(\rho L^d)-\tilde{\gamma}+o(1)\, ,
    \end{equation}
    where the $o(1)$ is uniform in $\mu$ and is understood as $L\to\infty$.
\end{proof}
\section{Proof in three dimensions}
\subsection{Non-Gaussian fluctuation limit}
The goal of this section is to prove the following proposition from which Proposition~\ref{prop:cltMain} follows.
\begin{proposition}\label{prop:clt}
    Assume that $\lambda_1>0$. Then, $\pn_L$ satisfies an infinitely divisible limit theorem at scale $\beta^{-1}L^2$ with infinitely divisible limiting distribution $\chi_\zeta$ where $\chi_\zeta$ has a characteristic function $\psi(t)$ given by
    \begin{equation}
        \log\psi(t)=\int_0^\infty \zeta_\L(u)\rk{\ex^{uit}-1-uit}\d u\, .
    \end{equation}
    Recall that $\zeta_\L(t)$ is defined in \eqref{eq:zetafu}. Furthermore
    \begin{equation}
        \log\psi(t)=-\log\det\nolimits_2\rk{\mathrm{Id}-it\Delta^{-1}}\, .
    \end{equation}
    Finally, the same limit theorem continues to hold when replacing $\pn_L$ with $\pn_L\hk{\le \gamma L^2}$ for $\gamma\to\infty$.
\end{proposition}
    Note that $\pn_L$ is of order $L^3$, hence the scaling relation is $2/3$.
\begin{proof}
    By the Campbell formula (see \cite[Theorem 3.9]{last2018lectures}) for $s\in\R$
    \begin{equation}
        \E\ek{\exp\rk{is\rk{\pn-\E[\pn]}}}=\exp\rk{\sum_{j=1}^{\rhoc L^3}\frac{1}{j}\rk{\ex^{isj}-1-isj}t_j}\, .
    \end{equation}
    Choose now $s=\beta tL^{-2}$ and observe that (recall \eqref{eq:specTjweights})
    \begin{equation}\label{eq:temp16041}
        \sum_{j=1}^{\rhoc L^3}\frac{1}{j}\rk{\ex^{isjL^{-2}}-1-isjL^{-2}}t_j=\sum_{j=1}^{\rhoc L^3}(\beta L^{-2})\frac{1}{\beta jL^{-2}}\rk{\ex^{it\beta jL^{-2}}-1-it\beta jL^{-2}}Z(\beta jL^{-2})\, .
    \end{equation}
    Note that as $t\to 0$ by \eqref{eq:MPexpansion} and \eqref{eq:zetafu}
    \begin{equation}
        t\mapsto \rk{\ex^{it}-1-it}\zeta_\Lambda(t)=\Ocal(t^2)\zeta_\Lambda(t)=\Ocal(t^2t^{-5/2})\, ,
    \end{equation}
    and is hence integrable. Furthermore, by Corollary \ref{cor:weightExp}, $u\mapsto \rk{\ex^{iu}-1-iu}\zeta_\Lambda(u)$ decays exponentially at infinity. Hence, by the definition of the Riemann integral and \eqref{eq:temp16041}
    \begin{equation}\label{eq:14620263}
        \lim_{L\to\infty}\sum_{j=1}^{\rhoc L^3}\frac{1}{j}\rk{\ex^{itjL^{-2}}-1-itjL^{-2}}t_j=\int_0^\infty \zeta_\L(u)\rk{\ex^{uit}-1-uit}\d u\, .
    \end{equation}
    To prove calculate the integral of $\zeta_\L(u)\rk{\ex^{uit}-1-uit}$, note that by Weyl's law, we can exchange integration and summation to write
    \begin{equation}
        \int_0^\infty \zeta_\L(u)\rk{\ex^{uit}-1-sit}\d u=\sum_{k\ge 1}\int_0^\infty \frac{\ex^{-\lambda_k u}}{u}\rk{\ex^{uit}-1-uit}\d u\, .
    \end{equation}
    Defining $I_k(t)=\int_0^\infty \frac{\ex^{-\lambda_k u}}{u}\rk{\ex^{uit}-1-uit}\d u$, we note that it satisfies
    \begin{equation}
        \frac{\d}{\d t}I_k(t)=\frac{i}{\lambda_k-it}-\frac{i}{\lambda_k}\, .
    \end{equation}
    $I_k(0)=0$ and hence
    \begin{equation}\label{eq:220420263}
        I_k(t)=-\ek{\log\rk{1-\frac{it}{\lambda_k}}+\frac{it}{\lambda_k}}\, .
    \end{equation}
    Note that from the Taylor expansion and Weyl's law, as $k\to\infty$
    \begin{equation}
        I_k(t)=-\frac{t^2}{\lambda^2_k}(1+o(1))=\Ocal\rk{\frac{t^2}{k^{4/3}}}\, .
    \end{equation}
    Hence,
    \begin{equation}\label{eq:220420261}
        \int_0^\infty \zeta_\L(u)\rk{\ex^{uit}-1-uit}\d u=-\sum_{k\ge 1}\log\rk{1-\frac{it}{\lambda_k}}+\frac{it}{\lambda_k}=-\log\prod_{k\ge 1}\ek{\rk{1-\frac{it}{\lambda_k}}\ex^{\frac{it}{\lambda_k}}}\, .
    \end{equation}
    The result then follows from the representation of the second regularized Fredholm determinant, see \cite[Theorem 9.2]{simon2005trace}. 
    
    Furthermore, the representation \eqref{eq:repreGammaFun} as an infinite sum of centered Gamma random variables holds true, as \eqref{eq:220420261} has exactly this representation of the characteristic function.

    Finally, note that in \eqref{eq:14620263}, we only used the fact that upper limit of the sum is $\gg L^2$. Hence, the statement also holds for $\pn_L\hk{\le \gamma L^2}$ with $\gamma\to\infty$.
\end{proof}
\subsection{Negative first boundary heat coefficient}\label{sec:Negativeboundaryelasticityconstant}
We first prove the result under the assumption $\lambda_1>0$. In the final subsubsection, we point out the slight modifications needed to treat the case $\lambda_1=0$.

The proof itself has three steps. First, we show that the expected contribution of loops below the mixing scale falls short of $\rhoc L^3$ by $2|a_1|\beta^{-1}L^2\log L$. Second, we prove that short-loop fluctuations are too expensive to compensate this difference. Third, we show that the missing mass is supplied by long loops of length $L^2\log L$, whose distribution is asymptotically governed by logarithmic combinatorial structures.
\subsubsection{The partition function and long loops}
The main strategy is as follows: the expected number of particles in the loop soup is $\rhoc L^3$ \textit{minus} a term proportional to $L^2\log(L)$. We then show that it is cheaper to supply the extra mass from long loops than from short loops. Controlling the probability distribution of the long loops up to order $1+o(1)$ yields the result.

Fix $\alpha=\alpha_L\in\N$ satisfying $\alpha\sim \log\log(L)/\log\log\log(L)$. Set
\begin{equation}
    \ps=\pn\hk{\le \alpha L^2}=\sum_{i=1}^n \ell(\omega_i)\1\gk{\ell(\omega_i)\le \alpha L^2}\, .
\end{equation}
\begin{lemma}\label{lem:calculationOfExpectation}
    The expectation of $\ps$ satisfies
    \begin{equation}
        \E\ek{\ps}=\rhoc L^3+\frac{2a_1L^2\log(L)}{\beta}+\Ocal\rk{L^2}\, .
    \end{equation}
\end{lemma}
Since $a_1<0$, $\E\ek{\ps}$ does not provide enough density to reach a density of $\rhoc$. The remaining weight will have to come from long loops, as we will see.
\begin{proof}
    Set
    \begin{equation}\label{eq:14620261}
        g(t)=Z(t)-a_0t^{-3/2}-a_1t^{-1}\, .
    \end{equation}
    By \eqref{eq:MPexpansion}, $g(t)=\Ocal(t^{-1/2})$ as $t\to 0$ and is hence integrable. Furthermore, $\int_0^bg(t)\d t=\Ocal(b^{1/2})$. Expand using \eqref{eq:14620261}
    \begin{equation}
        \E\ek{\pn\hk{\le L^2}}=a_0L^3\sum_{j=1}^{L^2}\frac{1}{(\beta j)^{3/2}}+a_1L^2\sum_{j=1}^{L^2}\frac{1}{(\beta j)}+L^1\sum_{j=1}^{L^2}g(\beta jL^{-2})\, .
    \end{equation}
    We first study the weight from loops with length $\le L^2$. Bounding the sum by an integral yields
    \begin{equation}
        a_0L^3\sum_{j=1}^{L^2}\frac{1}{(\beta j)^{3/2}}=(4\pi \beta)^{-3/2}\zeta(3/2)L^3-a_0L^3\sum_{j>L^2}\frac{1}{(\beta j)^{3/2}}=\rhoc L^3+\Ocal\rk{L^2}\, .
    \end{equation}
    Furthermore, the asymptotics of the harmonic series yield
    \begin{equation}
        a_1L^2\sum_{j=1}^{L^2}\frac{1}{(\beta j)}=\frac{2a_1L^2}{\beta}\log(L)+\Ocal\rk{L^2}\, .
    \end{equation}
    By the definition of the Riemann-integral
    \begin{equation}\label{eq:050520262}
        \sum_{j=1}^{L^2}g(\beta jL^{-2})\sim L^2\int_0^1 g(t)\d t=\Ocal\rk{L^2}\, .
    \end{equation}
    Finally, \eqref{eq:firstEvDom} implies
    \begin{equation}
        \E\ek{\pn\hk{\ge L^2}}\le C\sum_{j\ge L^2}\ex^{-\lambda_1 \beta jL^{-2}}=\Ocal\rk{L^2}\, .
    \end{equation}
    This concludes the proof.
\end{proof}
Next, a concentration bound for the short loops are proved.
\begin{lemma}\label{lem:upperboundd3negative}
    As $L\to\infty$, for every $r>0$ with $r\le \log\log\log(L)$
    \begin{equation}
        \P\rk{\abs{\ps-\E[\ps]}\ge \kappa L^2\log(L)}=\Ocal\rk{L^{-(r+\lambda_1\beta)\kappa(1+o(1))}}\, .
    \end{equation}
\end{lemma}
\begin{proof}
    Chebyshev's inequality and the Campbell formula (see \cite[Theorem 3.9]{last2018lectures}) imply that for every $t>0$
    \begin{equation}
        \P\rk{\abs{\ps-\E[\ps]}\ge\kappa L^2\log(L)}\le \ex^{-tL^2\log(L)\kappa}\exp\rk{\sum_{j=1}^{\alpha L^2}\frac{1}{j}\rk{\ex^{tj}-1-tj}t_j}\, .
    \end{equation}
    Choose $t=\beta\lambda_1L^{-2}+rL^{-2}$. Then, using \eqref{eq:tjForSmallj}
    \begin{equation}
        \sum_{j=1}^{L^2/r}\frac{1}{j}\rk{\ex^{tj}-1-tj}t_j\le C L^3\sum_{j=1}^{L^2/r}\frac{1}{j^{3/2+1}}\rk{\ex^{tj}-1-tj}\le C L^3t^2\sum_{j=1}^{L^2/r}j^{-1/2}=\Ocal(r)\, .
    \end{equation}
    Also
    \begin{equation}
        \sum_{j=L^2/r}^{L^2}\frac{1}{j}\rk{\ex^{tj}-1-tj}t_j\le C L^3 \ex^{r}\sum_{j=L^2/r}^{L^2} j^{-5/2}=o(\log(L))\, .
    \end{equation}
    Furthermore, by \eqref{eq:firstEvDom}
    \begin{equation}
        \sum_{j=L^2/r}^{\alpha L^2}\frac{1}{j}\rk{\ex^{tj}-1-tj}t_j\le C\sum_{j=L^2/r}^{\alpha L^2}\frac{\ex^{rL^{-2}}}{j}=\Ocal\rk{\frac{\ex^{r\alpha}}{\alpha}}=o\rk{\log(L)}\, ,
    \end{equation}
    by the assumption on $\alpha$. Combining the previous two equations yields
    \begin{equation}
         \P\rk{\abs{\ps-\E[\ps]}\ge\kappa L^2\log(L)}\le \ex^{-(r+\lambda_1\beta)\log(L)\kappa+o(\log(L))}\, .
    \end{equation}
\end{proof}
\begin{lemma}\label{lem:dowardsdeviations}
    There exists $c>0$ such that for all $1\le \alpha\le L$ and $\kappa>0$, as $L\to\infty$
    \begin{equation}
        \P\rk{\pn\hk{\le \alpha L^2}-\E\ek{\pn\hk{\le \alpha L^2}}\le -\kappa L^2\log(L)}=\Ocal\rk{\ex^{-c\kappa^3\log^3(L)}}
    \end{equation}
\end{lemma}
\begin{proof}
    As before, we can bound the event in question as
    \begin{equation}
        \P\rk{\pn\hk{\le \alpha L^2}-\E\ek{\pn\hk{\le \alpha L^2}}\le -\kappa L^2\log(L)}\le \ex^{-tL^2\log(L)\kappa}\exp\rk{\sum_{j=1}^{\alpha L^2}\frac{1}{j}\rk{\ex^{-tj}-1+tj}t_j}\, ,
    \end{equation}
    for $t>0$. Choose $t=rL^{-2}$ and bound
    \begin{equation}
        \sum_{j=1}^{1/t}\frac{1}{j}\rk{\ex^{-tj}-1+tj}t_j=\sum_{j=1}^{1/t}\frac{1}{j}\Ocal\rk{t^2 j^2}t_j=\Ocal\rk{r^{3/2}}\, .
    \end{equation}
    On the other hand
    \begin{equation}
        \sum_{j=1/t}^{L^2}\frac{1}{j}\rk{\ex^{-tj}-1+tj}t_j\le Ct \sum_{j=1/t}^{L^2}t_j=\Ocal\rk{r^{3/2}}\, .
    \end{equation}
    Finally
    \begin{equation}
         \sum_{j\ge L^2}\frac{1}{j}\rk{\ex^{-tj}-1+tj}t_j\le \Ocal\rk{r^{-1}}\, .
    \end{equation}
    Choosing $r=\delta\kappa^2\log^2(L)$ and making $\delta$ small enough yields the claim.
\end{proof}
\begin{lemma}\label{lem:distrMesco}
    Fix $M,\delta>0$. For $\alpha L^2\le k\le M\log(L)L^2$ with $k\ge \delta \log(L)L^2$
    \begin{equation}\label{eq:oneplusooneLong}
        \P\rk{\pmes=k}=\ex^{-\lambda_1 k\beta L^{-2}}\frac{p_1(1)}{L^2\alpha}(1+o(1))\, ,
    \end{equation}
    where $\pmes=\pn\hk{\alpha L^2\le \ell\le M\log(L)L^2}$ is the number of particles in loops of length between $\alpha L^2$ and $M\log(L)L^2$ and $p_1$ is the density of a probability distribution with Laplace transform given by
    \begin{equation}
        s\mapsto\exp\rk{-\int_0^1\frac{1-\ex^{-sx}}{x}\d x}\, .
    \end{equation}
    If $k/M\log(L)L^2\to y\in (1,\infty)$, \eqref{eq:oneplusooneLong} continues to hold, replacing $p_1(1)$ with $p_1(y)$.
    
    Finally, there exists $C>0$ such that for all $k$
    \begin{equation}\label{eq:uniformBoundLong}
        \P\rk{\pmes=k}\le C\frac{\ex^{-\lambda_1 (k+1)\beta L^{-2}}}{L^2\alpha}\, .
    \end{equation}
    The equation above also holds true for $\alpha=\mathrm{const}$.
\end{lemma}
The mesoscopic loops are already well mixed, as their duration is dominating $L^2$. They follow the same global law as loops of length proportional to $L^3=\abs{\Lambda_L}$.
\begin{proof}
By the independence of the Poisson point process
    \begin{equation}\label{eq:090520261}
        \P\rk{\pmes=k}=\P\rk{\pn\hk{\le k}=k}\ex^{-\sum_{j=k+1}^{ML^2\log(L)}\frac{t_j}{j}}\sim \P\rk{\pn\hk{\le k}=k}\ex^{-f(k)}\, ,
    \end{equation}
    where by Riemann approximation
    \begin{equation}
        f(k)\sim \int_{k/L^2}^\infty \frac{\ex^{-\beta \lambda_1 t}}{t}\d t=o(1)\, .
    \end{equation}
    Set $\Part(k,r)$ the set of all partitions of an integer $k$ where each member/part has size at least $r$, i.e.,
    \begin{equation}
        \Part(k,r)=\gk{\lambda\in\gk{r,r+1,\ldots}^\N\colon \sum_{l\ge 1}l\lambda_l=k}\, .
    \end{equation}
    By the Poisson representation
    \begin{equation}\label{eq:170420262}
        \P\rk{\pn\hk{\le k}=k}=\exp\rk{-\sum_{j=\alpha L^2}^{k}\frac{t_j}{j}}\sum_{\lambda\in \Part(k,\alpha L^2)}\prod_{r=\alpha L^2}^k\frac{t_r^{\lambda_r}}{\lambda_r!r^{\lambda_r}}\, .
    \end{equation}
   We would like to approximate $t_r\approx \ex^{-\lambda_1\beta rL^{-2}}$. Set
   \begin{equation}
       e(j)=\ex^{\lambda_1\beta jL^{-2}}Z(\beta j L^{-2})-1=\sum_{k\ge 2}\ex^{-(\lambda_k-\lambda_1)\beta jL^{-2}}=\Ocal\rk{\ex^{-(\lambda_2-\lambda_1)jL^{-2}}}\, ,
   \end{equation}
   as $L\to\infty$.
   Factoring gives
   \begin{equation}
       \sum_{\lambda\in \Part(k,\alpha L^2)}\prod_{r=\alpha L^2}^k\frac{t_r^{\lambda_r}}{\lambda_r!r^{\lambda_r}}=\ex^{-\lambda_1\beta kL^{-2}}\sum_{\lambda\in \Part(k,\alpha L^2)}\prod_{r=\alpha L^2}^k\frac{(1+e(r))^{\lambda_r}}{\lambda_r!r^{\lambda_r}}
   \end{equation}
 We can write now
 \begin{equation}
     \pn\hk{\le k}=\sum_{j=\alpha L^2}^{k}(X_j+Y_j)j\, ,
 \end{equation}
 where $X_j\sim\mathrm{Poi}(1/j)$ and $Y_j\sim \mathrm{Poi}(e(j)/j)$ are independent Poisson random variables. Hence, by a change of measure
    \begin{equation}\label{eq:2104292}
         \P\rk{\pn\hk{\le k}=k}=\ex^{-\lambda_1 k\beta L^{-2}}\ex^{-\sum_{j=\alpha L^2}^{k}j^{-1}(t_j-1-e(j))}\P\rk{\sum_{j=\alpha L^2}^{k}(X_j+Y_j)j=k}\, .
    \end{equation}
    Note that
    \begin{equation}
        \ex^{\sum_{j=\alpha L^2}^{k}1/j}=\frac{k}{\alpha L^2}(1+o(1))\, ,
    \end{equation}
    while
    \begin{equation}
        \ex^{\sum_{j=\alpha L^2}^{M\log(L)L^2}j^{-1}(t_j-e(j))}=1+o(1)\, ,
    \end{equation}
    which yields
    \begin{equation}
        \P\rk{\pn\hk{\le k}=k}\sim \frac{k}{\alpha L^2}\P\rk{\sum_{j=\alpha L^2}^{k}(X_j+Y_j)j=k}\, .
    \end{equation}
    Furthermore, using \cite[Eq. (4.8)]{arratia2003logarithmic}, for any $k\in \N$
    \begin{equation}\label{eq:050520261}
        \P\rk{\sum_{j=\alpha L^2}^{k}(X_j+Y_j)j=k}=\frac{1}{k}\sum_{l=\alpha L^2}^k\rk{1+e(l)}\P\rk{\sum_{j=\alpha L^2}^{k}(X_j+Y_j)j=k-l}\le \frac{C}{k}\, .
    \end{equation}
    This proves \eqref{eq:uniformBoundLong}.

    Set $\lambda_G=\lambda_2-\lambda_1$ and set $\widetilde{\alpha}=\log(\log(L)^{1.5})/\lambda_G$. Then
    \begin{equation}
        \P\rk{\sum_{j=\widetilde{\alpha}L^2}^{k}Y_j>0}=\exp\rk{-\sum_{j=\widetilde{\alpha}L^2}^{k}\frac{e(j)}{j}}\sim\exp\rk{-\frac{\ex^{-\lambda_G \widetilde{\alpha}}}{\widetilde{\alpha}}}\sim\exp\rk{-\frac{\log(L)^{1.5}}{\widetilde{\alpha}}}\, ,
    \end{equation}
    which decays faster than any polynomial. By the independence of $(X_j)_j$ and $(Y_j)_j$, we can hence neglect $Y_j$ for $j\ge \widetilde{\alpha}L^2$.
    
    Again by \cite[Eq. (4.8)]{arratia2003logarithmic}
    \begin{equation}
        \P\rk{\sum_{j=\alpha L^2}^{\widetilde{\alpha}L^2}Y_jj =k}\le C\frac{\ex^{-\lambda_G \alpha}}{k}\sum_{l=\alpha L^2}^{\widetilde{\alpha}L^2}\P\rk{\sum_{j=\alpha L^2}^{\widetilde{\alpha}L^2}Y_jj =k-l}\, .
    \end{equation}
    Note that for $k$ larger than $2\tilde{\alpha}L^2$
    \begin{equation}
        \sum_{l=\alpha L^2}^{\widetilde{\alpha}L^2}\P\rk{\sum_{j=\alpha L^2}^{\widetilde{\alpha}L^2}Y_jj =k-l}\le \P\rk{\sum_{j=\alpha L^2}^{\widetilde{\alpha}L^2}Y_jj\ge k-\widetilde{\alpha}L^2}\le C\frac{\ex^{-\alpha \lambda_G}\widetilde{\alpha} L^2}{k}\, .
    \end{equation}
    Hence
    \begin{equation}\label{eq:210420262}
        \P\rk{\sum_{j=\alpha L^2}^{\widetilde{\alpha}L^2}Y_jj =k}\le C\frac{\ex^{-2\lambda_G \alpha}\widetilde{\alpha} L^2}{k^2}\, .
    \end{equation}
Write
    \begin{equation}
        T_X=\sum_{j=\alpha L^2}^{k}jX_j\qquad\textnormal{and}\qquad T_Y=\sum_{j=\alpha L^2}^{\widetilde{\alpha} L^2}jY_j\, .
    \end{equation}
    
    By \cite[Theorem 4.13]{arratia2003logarithmic}
    \begin{equation}\label{eq:210420261}
        \P\rk{T_X=k}\sim \frac{p_1(1)}{k}\, .
    \end{equation}
    Take $\e_L\to 0$ slowly. Then
    \begin{equation}
    \begin{split}
         \P\rk{T_X+T_Y=k}=&\P\rk{T_X+T_Y=k,T_Y\le \e_L L^2\log(L)}\\
         &+\sum_{l=\e_L L^2\log(L)}^{ L^2\log(L)}\P\rk{T_X=k-l,T_Y=l}\, .
    \end{split}
    \end{equation}
    The independence between $(X_j)_j$ and $(Y_j)_j$ and \eqref{eq:210420261} imply that
    \begin{equation}
        \P\rk{T_X+T_Y=k,T_Y\le \e_L L^2\log(L)}\sim \frac{p_1(1)}{k}\, .
    \end{equation}
    By \eqref{eq:210420262}
    \begin{equation}
        \sum_{l=\e_L L^2\log(L)}^{ L^2\log(L)}\P\rk{T_X=k-l,T_Y=l}\le C\frac{\widetilde{\alpha}\ex^{-2\lambda_G\alpha}}{L^2\log^2(L)\e_L^2}\, .
    \end{equation}
    Choosing $\e_L=\ex^{-\lambda_G\alpha}$ and noting that $\widetilde{\alpha}=o(\log(L))$ yields
    \begin{equation}\label{eq:220620261}
        \P\rk{T_X+T_Y=k}\sim \frac{p_1(1)}{k}\, .
    \end{equation}
    Recalling \eqref{eq:210420262}, this implies that
    \begin{equation}
         \P\rk{\pmes=k}\sim \ex^{-\lambda_1 k\beta L^{-2}}\ex^{-\sum_{j=\alpha L^2}^{k}j^{-1}(t_j-1-e(j))}\frac{p_1(1)}{k}\, .
    \end{equation}
    
    This implies
    \begin{equation}
         \P\rk{\pmes=k}=\ex^{-\lambda_1 k\beta L^{-2}}\frac{p_1(1)}{L^2\alpha}(1+o(1))\, .
    \end{equation}
    This concludes the proof of the main statement.

    To show that if $k/M\log(L)L^2\to y\in (1,\infty)$, \eqref{eq:oneplusooneLong} continues to hold, replacing $p_1(1)$ with $p_1(y)$, the only modification needed is \eqref{eq:210420261}. Indeed, it was shown in \cite[Theorem 4.13]{arratia2003logarithmic} that
    \begin{equation}
        \P\rk{\sum_{j=\alpha L^2}^{M\log(L)L^2} jX_j=y ML^2\log(L)}\sim \frac{p_1(y)}{k}\, .
    \end{equation}
    The rest of the proof carries over exactly in the same manner. Also note that $p_1(y)=p_1(1)$ for $y\in (0,1]$.
\end{proof}
\begin{lemma}\label{lem:partdoubleneg}
    As $L\to\infty$
    \begin{equation}
        \P\rk{\pn=\rhoc L^3}= L^{-2+2a_1\lambda_1(1+o(1))}\, .
    \end{equation}
    The $o(1)$-term can be written as $\Ocal(\log\log\log(L)/\log(L))$.
\end{lemma}
\begin{proof}
\textbf{Upper bound:}
\noindent
    Choosing $r=\log\log\log(L)$ and $\kappa=M/r$ in Lemma \ref{lem:upperboundd3negative} yields
    \begin{equation}\label{eq:020520261}
        \P\rk{\abs{\ps-\E[\ps]}\ge \kappa L^2\log(L)}=\Ocal\rk{L^{-M/2}}\, .
    \end{equation}
    Choosing $M>0$ large enough shows that 
    \begin{equation}
        \P\rk{\pn=\rhoc L^3}\sim \P\rk{\pn=\rhoc L^3\mid \abs{\ps-\E[\ps]}\le \kappa L^2\log(L)}\, .
    \end{equation}
    By Lemma~\ref{lem:calculationOfExpectation}, this implies
    \begin{equation}
         \P\rk{\pn=\rhoc L^3}\le \P\rk{\pn=\rhoc L^3,\pn\hk{\ge \alpha L^2}\sim -2a_1 L^2\log(L)/\beta \mid \abs{\ps-\E[\ps]}\le \kappa L^2\log(L)}\, .
    \end{equation}
    Observe that for every $c>0$, there exists $M>0$ such that
    \begin{equation}
        \P\rk{\pn\hk{\ge M\log(L)L^2}>0}=1-\exp\rk{-\sum_{j\ge ML^2\log(L)}\frac{t_j}{j}}\le 1- \ex^{-cL^{-c}}=\Ocal\rk{L^{-c}}\, .
    \end{equation}
    Choosing $c>0$ sufficiently large and using Lemma~\ref{lem:distrMesco} then gives
    \begin{equation}
        \P\rk{\pn=\rhoc L^3} \le L^{-2+2a_1\lambda_1+\Ocal(\log\log\log(L)/\log(L))}\, .
    \end{equation}
    \textbf{Lower bound:} On account of Proposition~\ref{prop:clt} and Lemma~\ref{lem:calculationOfExpectation}, there exists $M>0$
    \begin{equation}
        \P\rk{\ps\in \rhoc L^3+2a_1L^2\log(L)/\beta+[-ML^2,ML^2]}\ge \frac{1}{2}\, .
    \end{equation}
    Hence
    \begin{equation}
         \P\rk{\pn=\rhoc L^3} \ge \frac{1}{2}\P\rk{\pn\hk{\ge \alpha L^2}\sim -2a_1 L^2\log(L)/\beta}\, .
    \end{equation}
    The proof then concludes identically to the upper bound.
\end{proof}
\subsubsection{The one-particle reduced density matrix}
For the remainder of this section, fix $\e>0$ and always assume that $x,y$ satisfy $\dist\rk{x,\partial \L_L}>\e L$ and $\dist\rk{y,\partial \L_L}>\e L$. This implies that by the spectral theorem that as $tL^{-2}\to\infty$
\begin{equation}\label{eq:expansionpoint}
    p_{t}\hk{L}(x,y)\sim \phi\hk{1}_L(x)\phi\hk{1}_L(y)\ex^{-\lambda_1 t L^{-2}}(1+o(1))\, .
\end{equation}
Furthermore, note that by Brownian scaling that
\begin{equation}
    \phi\hk{1}_L(x)=\phi\hk{1}(x/L)L^{-3/2}\, .
\end{equation}
Assume that $x,y$ are well-separated: $\dist(x,y)>\e L$.

Fix a sequence $\alpha_1=o(\log(L))$ with $\alpha_1\to\infty$ and $\alpha_1\in\N$. Recall the representation of the one-particle reduced density matrix
\begin{equation}
     \gamma_{L,\beta,\rhoc}(x,y)=\sum_{r=1}^{\rhoc L^3}\frac{p_{r\beta }\hk{L}(x,y)\P_{L,\beta,\rhoc}\rk{\pn_L=\rhoc L^3-r}}{\P_{L,\beta,\rhoc}\rk{\pn_L=\rhoc L^3}}\, .
\end{equation}
The idea of the proof is now to show that by mixing, it is essentially equally likely to have a long closed loop compared to a long open path from $x$ to $y$. 

\noindent
\textbf{Concentrating the small loops}: choose $\delta=\tfrac{M}{\log\log\log(L)}$ for some $M>0$ adjusted later. As in Lemma~\ref{lem:partdoubleneg} and write
\begin{equation}
    D_+=\ek{\E\ek{\pn\hk{\le L^2}}-\delta L^2\log(L),\E\ek{\pn\hk{\le L^2}}+\delta L^2\log(L)}\, .
\end{equation}
Expand
\begin{multline}
    \P_{L,\beta,\rhoc}\rk{\pn_L=\rhoc L^3-r}=\P_{L,\beta,\rhoc}\rk{\pn_L=\rhoc L^3-r,\ps\in D_+}\\
    +\P_{L,\beta,\rhoc}\rk{\pn_L=\rhoc L^3-r,\ps\notin D_+}\, .
\end{multline}
As in \eqref{eq:020520261}, the sum over the term $\P_{L,\beta,\rhoc}\rk{\pn_L=\rhoc L^3-r,\ps\notin D_+}$ can be neglected.

\noindent
\textbf{Breaking-open long loops}: write $\up=(-2a_1\beta^{-1}+\delta)L^2\log(L)$. By the previous paragraph
\begin{equation}\label{eq:14620262}
     \gamma_{L,\beta,\rhoc}(x,y)\sim\sum_{r=1}^{\up}\frac{p_{\beta r}\hk{L}(x,y)\P_{L,\beta,\rhoc}\rk{\pn_L=\rhoc L^3-r,\ps\in D_+}}{\P_{L,\beta,\rhoc}\rk{\pn_L=\rhoc L^3}}\, .
\end{equation}
Expand
\begin{equation}
    \P_{L,\beta,\rhoc}\rk{\pn_L=\rhoc L^3-r,\ps\in D_+}=\sum_{k\in D_+}\P_{L,\beta,\rhoc}\rk{\pn\hk{\le L^2}=k,\pmes=\rhoc L^3-r-k}\, .
\end{equation}
Independence of $\pn\hk{\le L^2}$ and $\pmes$ implies that that the probabilities can be decoupled. To calculate the numerator of \eqref{eq:14620262}, it is hence key to estimate
\begin{equation}
    \sum_{r=1}^{\up}p_{\beta r}\hk{L}(x,y)\P_{L,\beta,\rhoc}\rk{\pmes=\rhoc L^3-r-k}\, .
\end{equation}

\noindent
\textbf{Main contribution}: fix $\e>0$ and write
\begin{equation}
    D_\main=\ek{\e\log(L)L^2,(-2a_1\beta^{-1}+\delta-\e)\log(L)L^2}\, .
\end{equation}
If $r\in D_\main$, then also $\rhoc-r-k\ge \e\log(L)L^2$ and by Lemma~\ref{lem:distrMesco} and \eqref{eq:expansionpoint}
\begin{equation}
\begin{split}
    \sum_{r\in D_\main}&p_{\beta r}\hk{L}(x,y)\P_{L,\beta,\rhoc}\rk{\pmes=\rhoc L^3-r-k}\\
    &\sim L^{-3}\phi\hk{1}(x/L)\phi\hk{1}(y/L)\sum_{r\in D_\main}\ex^{-\beta\lambda_1rL^{-2}}\ex^{-\beta\lambda_1(\rhoc L^3-k-r)L^{-2}}\\
    &\sim L^{-3}\phi\hk{1}(x/L)\phi\hk{1}(y/L)\sum_{r\in D_\main}\P_{L,\beta,\rhoc}\rk{\pmes=\rhoc L^3-k}\\
    & \sim \phi\hk{1}(x/L)\phi\hk{1}(y/L)(-2a_1\beta^{-1}+\delta-2\e)L^{-1}\log(L)\P_{L,\beta,\rhoc}\rk{\pmes=\rhoc L^3-k}\, .
\end{split}
\end{equation}
Hence
\begin{equation}
    \begin{split}
        \sum_{k\in D_+}&\sum_{r\in D_\main}\P_{L,\beta,\rhoc}\rk{\pn\hk{\le L^2}=k,\pmes=\rhoc L^3-r-k}\\
        &\sim \phi\hk{1}(x/L)\phi\hk{1}(y/L)(-2a_1\beta^{-1}+\delta-2\e)L^{-1}\log(L)\sum_{k\in D_+}\P_{L,\beta,\rhoc}\rk{\pn\hk{\le L^2}=k,\pmes=\rhoc L^3-k}\\
        &\sim \phi\hk{1}(x/L)\phi\hk{1}(y/L)(-2a_1\beta^{-1}+\delta-2\e)L^{-1}\log(L)\P_{L,\beta,\rhoc}\rk{\pn_L=\rhoc L^3}\, .
    \end{split}
\end{equation}
This implies that
\begin{multline}
     \sum_{k\in D_+}\sum_{r\in D_\main}\frac{\P_{L,\beta,\rhoc}\rk{\pn\hk{\le L^2}=k,\pmes=\rhoc L^3-r-k}}{\P_{L,\beta,\rhoc}\rk{\pn=\rhoc L^3}}\\
     \sim \phi\hk{1}(x/L)\phi\hk{1}(y/L) (-2a_1\beta^{-1}+\delta-2\e)L^{-1}\log(L)\, .
\end{multline}
This concludes the estimation of the main contribution.

\noindent
\textbf{Error Analysis}: By \cite[Theorem 6.3.8]{a1981operator}
\begin{equation}
     \begin{split}
         \sum_{r=1}^{L^2}p_{\beta r}\hk{L}(x,y)&\P_{L,\beta,\rhoc}\rk{\pmes=\rhoc L^3-r-k}\\
         &\le C\sum_{r=1}^{L^2}r^{-d/2}\ex^{-c L^2/r}\P_{L,\beta,\rhoc}\rk{\pmes=\rhoc L^3-r-k}\\
         &\le C\P_{L,\beta,\rhoc}\rk{\pmes=\rhoc L^3-k}\sum_{r=1}^{L^2}r^{-d/2}\ex^{-c L^2/r}\\
         &\le CL^{-1}\P_{L,\beta,\rhoc}\rk{\pmes=\rhoc L^3-k}\, ,
     \end{split}
\end{equation}
hence, the sum from $r=1$ to $r=L^2$ is negligible.

Using the same decomposition as for the main contribution and the uniform bound from \eqref{eq:uniformBoundLong}, give the following
\begin{equation}
    \sum_{k\in D_+}\sum_{r=L^2}^{\e \log(L)L^2}\frac{p_{\beta r}\hk{L}(x,y)\P_{L,\beta,\rhoc}\rk{\ps=k,\pmes=\rhoc L^3-r-k}}{\P_{L,\beta,\rhoc}\rk{\pn=\rhoc L^3}}\le C \e L^{-1}\log(L)\, ,
\end{equation}
as well as
\begin{equation}
     \sum_{k\in D_+}\sum_{r=(-2a_1\beta^{-1}+\delta-\e) \log(L)L^2}^{\up}\frac{p_{\beta r}\hk{L}(x,y)\P_{L,\beta,\rhoc}\rk{\ps=k,\pmes=\rhoc L^3-r-k}}{\P_{L,\beta,\rhoc}\rk{\pn=\rhoc L^3}}\le C \e L^{-1}\log(L)\, .
\end{equation}
This concludes the proof of \eqref{eq:themoprdm}.
\subsubsection{The Poisson--Dirichlet scaling limit}
Given the previous results, it is standard to prove the statement via convergence of finite marginals, see \cite{konig2025off}. 

Set
\begin{equation}
    C\hk{j} := \#\gk{\text{loops of length } j}, \qquad  \pn\hk{k}:=kC_k .
\end{equation}
Fix an index $M\in \N$ and $1>y_1>y_2>\ldots>y_M>0$ such that $\sum_i y_i=\Sigma_y<1$. Set $\mass=-2a_1 L^2\log(L)/\beta$. Set
\begin{equation}
    A_{y}=\gk{N\hk{>y_1\mass}=0}\cap\bigcap_{i=1}^M\gk{C\hk{y_i\mass}=1}\cap \bigcap_{\mycom{j=y_M\mass+1}{j\neq y_i\mass}}^{y_1\mass-1}\gk{C\hk{j}=0}\cap \gk{C\hk{<y_M\mass}=\rhoc L^3-\Sigma_y\mass}\, ,
\end{equation}
where we assume without loss of generality that $y_i\mass\in \N$ for all $i$.

Note that by the Poisson representation all the events above are independent. By \eqref{eq:090520261}, $\gk{N\hk{>y_1\mass}=0}$ occurs with probability $1+o(1)$. On the other hand, using \eqref{eq:spectraldec}
\begin{equation}
    \P\rk{\bigcap_{i=1}^M\gk{C\hk{y_i\mass}=1}}=\ex^{-\sum_{i=1}^M\tfrac{t_{y_i\mass}}{y_i\mass}}\prod_{i=1}^M\frac{t_{y_i\mass}}{y_i\mass}\sim \frac{\ex^{-\lambda_1\beta \Sigma_y\mass L^{-2}}}{\mass^M\prod_{i=1}^M y_i}\, .
\end{equation}
As before,
\begin{equation}
    \P\rk{\bigcap_{\mycom{j=y_M\mass+1}{j\neq y_i\mass}}^{y_1\mass-1}\gk{C\hk{j}=0}}=\exp\rk{-\sum_{\mycom{j=y_M\mass+1}{j\neq y_i\mass}}^{y_1\mass-1}\frac{t_j}{j} }=1-o(1)\, .
\end{equation}
Finally, by Lemma~\ref{lem:distrMesco} and restricting the short loops to $L^2\log(L)/\log\log\log(L)$, we obtain
\begin{equation}
    \P\rk{\pn\hk{<y_M\mass}=\rhoc L^3-\Sigma_y\mass}\sim \frac{p_1\rk{\frac{1-\Sigma_y}{y_M}}\ex^{-\lambda_1\beta (1-\Sigma_y)\mass L^{-2}}}{L^2\alpha}\, .
\end{equation}
This implies that
\begin{equation}
    \P\rk{\Lcal_1=y_1\mass,\ldots,\Lcal_m=y_m\mass|\pn_L=\rhoc L^3}\sim\frac{p_1\rk{\frac{1-\Sigma_y}{y_M}}}{\mass^M}\, . 
\end{equation}
Since $M\in\N$ was arbitrary, the convergence follows from Scheff\'e's theorem \cite[Corollary 5.11]{arratia2003logarithmic}.
\subsubsection{Negative first boundary heat coefficient, zero first eigenvalue}
The approach is similar to the previous section, and hence we will supply less intermediate steps for the calculations. The main strategy is as before, however, an additional penalty needs to be added to account for the suppression of loops with length $\ge -2a_1\log(L)L^2/\beta$. As the result is different, we prove a variant of Lemma~\ref{lem:distrMesco}.
\begin{lemma}\label{lem:distMetro0}
    Fix $M,\delta>0$. For $\alpha L^2\le k\le M\log(L)L^2$ with $k\ge \delta \log(L)L^2$
    \begin{equation}\label{eq:oneplusooneLong0}
        \P\rk{\pmes=k}=\frac{p_1(1)}{M\log(L)L^2}(1+o(1))\, .
    \end{equation}
    If $k/M\log(L)L^2\to y\in (1,\infty)$, \eqref{eq:oneplusooneLong} continues to hold, replacing $p_1(1)$ with $p_1(y)$.
    
    Finally, there exists $C>0$ such that for all $k$
    \begin{equation}\label{eq:uniformBoundLong0}
        \P\rk{\pmes=k}\le \frac{C}{M\log(L)L^2}\, .
    \end{equation}
    The equation above also holds true for $\alpha=\mathrm{const}$.
\end{lemma}
\begin{proof}
    We follow the outline of the proof of Lemma~\ref{lem:distrMesco} and the notation introduced therein: using the Poisson property, we write
    \begin{equation}
        \P\rk{\pmes=k}=\P\rk{\pn\hk{\le k}=k}\ex^{-\sum_{j=k+1}^{ML^2\log(L)}\frac{t_j}{j}}\, .
    \end{equation}
    However, for $\lambda_1=0$, the second term does give an additional factor:
    \begin{equation}
        \ex^{-\sum_{j=k+1}^{ML^2\log(L)}}=\frac{k}{M^2\log(L)}(1+o(1))\, .
    \end{equation}
    This yields, as in \eqref{eq:2104292},
    \begin{equation}
       \P\rk{\pmes=k}\sim \frac{k}{M^2\log(L)} \ex^{-\sum_{j=\alpha L^2}^{k}j^{-1}(t_j-1-e(j))}\P\rk{\sum_{j=\alpha L^2}^{k}(X_j+Y_j)j=k}\, .
    \end{equation}
    The asymptotics of the probability of $\sum_{j=\alpha L^2}^{k}(X_j+Y_j)j=k$ depend only on the spectral gap (see \eqref{eq:220620261} and its derivation) and therefore we find that
    \begin{equation}
        \begin{split}
             \P\rk{\pmes=k}\sim \frac{k}{M^2\log(L)} \ex^{-\sum_{j=\alpha L^2}^{k}j^{-1}(t_j-1-e(j))}\frac{p_1(1)}{k}\sim \frac{p_1(1)}{M^2\log(L)}\, .
        \end{split}
    \end{equation}
\end{proof}
\textbf{Proof of Theorem~\ref{thm:bothneg} in the case $\lambda_1=0$}: on account of Lemma~\ref{lem:dowardsdeviations}, the event
\begin{equation}
    \gk{\pn-\E[\pn]\le -\kappa L^2\log(L)}\, ,
\end{equation}
is negligible for $\kappa=o(1)$ going to zero sufficiently slowly. Hence
\begin{equation}
    \P\rk{\pn=\rhoc L^3}\sim \P\rk{\pn=\rhoc L^3,\pn\hk{\ge \alpha L^2}\sim -2a_1 L^2\log(L)/\beta}\, . 
\end{equation}
Estimating the probability of $\gk{\pn\hk{\ge \alpha L^2}\sim -2a_1 L^2\log(L)/\beta}$ requires more care, but the independence of the different loop lengths readily resolves the issue:
\begin{equation}
    \P\rk{\pn\hk{\ge \alpha L^2}\sim -2a_1 L^2\log(L)/\beta}=\P\rk{\pmes\sim -2a_1 L^2\log(L)/\beta}\P\rk{\pn\hk{\ge -2a_1 L^2\log(L)/\beta}=0}\, .
\end{equation}
By the Poisson property and recalling that $\lambda_1=0$ implies that $t_j\sim 1$ for $j\gg L^2$,
\begin{equation}\label{eq:260420261}
    \P\rk{\pn\hk{\ge -2a_1 L^2\log(L)/\beta}=0}=\exp\rk{-\sum_{j=-2a_1 L^2\log(L)/\beta}^{\rhoc L^3}\frac{t_j}{j}}\sim -\frac{2a_1}{\rhoc \beta}\frac{\log(L)}{L}\, .
\end{equation}
Using Lemma~\ref{lem:distMetro0}, we find that
\begin{equation}
    \P\rk{\pn\hk{\ge \alpha L^2}\sim -2a_1 L^2\log(L)/\beta}\sim \rhoc^{-1}L^{-3}\, .
\end{equation}

The limit theorem in Proposition~\ref{prop:cltMain} fails for $\lambda_1=0$. However, for every fixed $\alpha>0$, the CLT does hold for $\pn\hk{\le \alpha L^2}$. This allows us to follow the proof of Lemma~\ref{lem:partdoubleneg} and find that partition function is given by
\begin{equation}
    \P\rk{\pn=\rhoc L^3}=L^{-3+o(1)}\, .
\end{equation}
Similar adaptations yield the claimed result for the one-particle reduced density matrix, where the additional factor from the suppression cancels out in numerator and denominator.
\subsection{Positive first boundary heat coefficient}
We first prove the result under the assumption $\lambda_1>0$. In the final sub-subsection, we point out the slight modifications needed to treat the case $\lambda_1=0$.
\subsubsection{The partition function and long loops}
Lemma~\ref{lem:calculationOfExpectation} still holds true. However, the sign of $a_1$ is reversed; therefore, $\ps$ already exceeds the threshold of $\rhoc L^3$ by a factor of $2a_1L^2\log(L)/\beta$, which requires a downward correction via a chemical potential, different from the case $a_1<0$. 
\begin{lemma}\label{lem:changeOfDens}
    Set for
    \begin{equation}
        \rho(\mu)=\sum_{j=1}^{\rhoc L^3}\ex^{\beta\mu j}t_j\, .
    \end{equation}
    For $\mu=-r^2L^{-2}\log^{2}(L)$,
    \begin{equation}
         \rho(\mu)= \rho(0)-\frac{r\log(L)L^2}{4\pi\beta }\rk{1+o(1)}
    \end{equation}
The $o(1)$-term can be chosen $\log^{-3/4}(L)$.
\end{lemma}
\begin{proof}
    Set $\mu^+=L^{-2}/\log^{-1/2}(L)$, an additional cut-off parameter. Then, $1/\mu\ll 1/\mu^+\ll L^2$. Expand using Corollary~\ref{cor:weightExp}
    \begin{equation}
        \sum_{j=1}^{1/\mu^+}(1-\ex^{\beta \mu j})t_j=\sum_{j=1}^{1/\mu^+}(1-\ex^{\beta \mu j})\rk{a_0(\beta j)^{-3/2}L^3+\Ocal\rk{j^{-1}L^2}}\, .
    \end{equation}
    Using a Riemann approximation, it can be seen that the first term equals
    \begin{equation}
    \begin{split}
    a_0L^3\beta^{-1}\abs{\mu}^{1/2}\sum_{j=1}^{1/\mu^+}\beta\abs{\mu}\frac{(1-\ex^{-\beta \abs{\mu} j})}{(\beta \abs{\mu} j)^{3/2}} &=4^{-3/2}\pi^{-3/2}rL^2\log(L)\int_0^\infty \frac{1-\ex^{-x}}{x^{3/2}}\d x(1+o(1))\\
      &=\frac{r\log(L)L^2}{4\pi \beta}\, ,
          \end{split}
    \end{equation}
    where the computation of $\int_0^\infty \frac{1-\ex^{-x}}{x^{3/2}}\d x=4^{1/2}\pi^{1/2}$ can be achieved via the Gamma function.
    
    Furthermore,
    \begin{equation}
        L^3\sum_{j=1/\mu^+}^{L^2}\frac{1}{j^{3/2}}=\Ocal\rk{L^2\log^{1/4}(L)}\, .
    \end{equation}
    On the other hand, using the Taylor approximation
    \begin{equation}
        L^2\sum_{j=1}^{1/\mu}(1-\ex^{\beta \mu j})j^{-1}=\Ocal(L^2)\, ,
    \end{equation}
    and
    \begin{equation}
         L^2\sum_{j=1/\mu}^{1/\mu^+}(1-\ex^{\beta \mu j})j^{-1}=\Ocal\rk{L^2\log(\mu/\mu^+)}=\Ocal\rk{L^2\log\log(L)}\, .
    \end{equation}
    Finally,
    \begin{equation}
        L^2\sum_{j=1/\mu^+}^{ L^2}\frac{1}{j}=\Ocal\rk{L^2}\, .
    \end{equation}
    Finally, as seen before, due to the exponential decay from \eqref{eq:firstEvDom}
    \begin{equation}
        \sum_{j\ge L^2}\rk{\ex^{\beta \mu j}+1}t_j=\Ocal\rk{L^2}\, .
    \end{equation}
    This concludes the proof.
\end{proof}
Denote now by $\P_\mu=\P_{L,\beta,N,\mu}$ the Poisson point process with intensity measure
\begin{equation}
    M_{L,\beta,N,\mu}=\sum_{j=1}^{N}\frac{\ex^{\beta \mu j}}{j}\int_{\Lambda_L}\d x \P_{x,x,\beta j}\hk{L}\, .
\end{equation}
\begin{lemma}\label{lem:localCLT}
    Under $\P_\mu$, $\pn$ satisfies a \textnormal{local} central (Gaussian) limit theorem. Let \(\mu=-r^2L^{-2}\log^2L\) and
\[
    \nu=\frac{\sqrt{\log L}}{L^2}.
\]
Then, uniformly for integer \(n\) such that
\[
    x_L=\nu(n-\E_\mu[\mathscr N])
\]
stays in a compact subset of \(\mathbb R\),
\[
    \nu^{-1}\P_\mu(\mathscr N=n)
    =
    f_\sigma(x_L)+o(1),
    \qquad
    \sigma^2=\frac{1}{8\pi\beta^2r}.
\]
    where $f_{\sigma}(x)$ is the density of a Gaussian random variable with variance $\sigma^{-2}=8\pi \beta^2 r$. Note that we assume without loss of generality that $x, L, \mu$ is chosen up to $(1+o(1))$ such that the event above has a positive probability.
\end{lemma}
\begin{proof}
    We set $x=0$, as the general case follows analogously.

    Denote
    \begin{equation}
        X_L=\nu\rk{\pn-\E_\mu[\pn]}\, .
    \end{equation}
    Then,
    \begin{equation}
        \P_\mu\rk{X_L=0}=\frac{1}{2\pi}\int_{-\pi}^{\pi}\E\ek{\ex^{itX_L}}\d t=\frac{1}{2\pi}\int_{-\pi}^{\pi}\exp\rk{\sum_{j=1}^{\rhoc L^3}\frac{\ex^{\beta \mu j}}{j}\rk{\ex^{it j}-1-itj}t_j}\d t\, .
    \end{equation}
    Changing variables yields
    \begin{equation}
        \P_\mu\rk{X_L=0}=\frac{\nu}{2\pi}\int_{-\pi/\nu}^{\pi/\nu}\exp\rk{\sum_{j=1}^{\rhoc L^3}\frac{\ex^{\beta \mu j}}{j}\rk{\ex^{it j\nu}-1-it\nu j}t_j}\d t\, .
    \end{equation}
    First, values of large $j$ can be suppressed: for $\e>0$ small enough, set
    \begin{equation}
        M=\frac{\log^{1+\e}(L)}{\mu}=\frac{L^2}{\log^{1-\e}(L)}\, .
    \end{equation}
    Then, for $t=\Ocal(\nu^{-1})$ 
    \begin{equation}
        \sum_{j=M}^{\rhoc L^3}\frac{\ex^{\beta \mu j}}{j}\rk{\ex^{it j\nu}-1-it\nu j}t_j=o\rk{\ex^{-\beta\log^{1+\e/2}(L)}}\, .
    \end{equation}
    Hence, it suffices to study the sum from $j=1$ up to $M$. Furthermore, for $\abs{t}=\Ocal\rk{\log^{1/2-2\e}(L)}$, $\abs{Mt}=o\rk{1/\nu}$ and hence
    \begin{equation}
        \sum_{j=1}^{M}\frac{\ex^{\beta \mu j}}{j}\rk{\ex^{it j\nu}-1-it\nu j}t_j\sim -\frac{t^2\nu^2}{2}\sum_{j=1}^{M}\ex^{\beta \mu j}jt_j\, ,
    \end{equation}
    uniformly in $L$ and $t$. Next, expand
    \begin{equation}
        \sum_{j=1}^{M}\ex^{\beta \mu j}jt_j=\sum_{j=1}^{M}\ex^{\beta \mu j}j\rk{(4\pi\beta)^{-3/2}L^3j^{-3/2} +\Ocal\rk{L^2 j^{-1}}}\, .
    \end{equation}
    Using the Euler--Maclaurin formula
    \begin{equation}
        \sum_{j=1}^{M}\ex^{\beta \mu j}j^{-1/2}=\pi^{1/2}(-\beta \mu)^{-1/2}+\Ocal(1)=\pi^{1/2}\beta^{-1/2}r^{-1}\frac{L}{\log(L)}+\Ocal(1)\, .
    \end{equation}
    Furthermore
    \begin{equation}
        L^2 \sum_{j=1}^{M}\ex^{\beta \mu j}=\Ocal\rk{L^2\mu^{-1}}=\Ocal\rk{\frac{L^4}{\log^2(L)}}\, .
    \end{equation}
    Hence
    \begin{equation}
        \sum_{j=1}^{M}\ex^{\beta \mu j}jt_j=\frac{1}{8\pi\beta^2r}\frac{L^4}{\log(L)}(1+o(1))\, ,
    \end{equation}
    and thus
    \begin{equation}
        \sum_{j=1}^{M}\frac{\ex^{\beta \mu j}}{j}\rk{\ex^{it j\nu}-1-it\nu j}t_j\sim -\frac{t^2}{16\pi\beta^2r}\, .
    \end{equation}
    Using the same bounds, one finds for some $c>0$ and $\abs{t\nu }<\pi$
    \begin{equation}
         \Re\rk{\sum_{j=1}^{M}\frac{\ex^{\beta \mu j}}{j}\rk{\ex^{it j\nu}-1-it\nu j}t_j}=-\sum_{j=1}^{M}\frac{\ex^{\beta \mu j}}{j}\rk{1-\cos(t\nu j)}t_j\le -c t^2\, .
    \end{equation}
    This shows that for some $\delta>0$
    \begin{equation}
         \P_\mu\rk{X_L=0}=\frac{\nu}{2\pi}\rk{\sqrt{2\pi^2\beta^2 8r}+\Ocal\rk{L^{-\delta}}}\, .
    \end{equation}
\end{proof}
\begin{proposition}\label{prop:posboundarypart}
    As $L\to \infty$
    \begin{equation}
        \log \P\rk{\pn=\rhoc L^3}={-128\pi^2a_1^3\log^3(L)/3(1+o(1))}\, .
    \end{equation}
\end{proposition}
\begin{proof}
Define
\begin{equation}
    p(\mu)=\sum_{j=1}^{\rhoc L^3}\frac{\ex^{\beta \mu j}}{j}t_j\, ,
\end{equation}
the pressure of the Bose gas at chemical potential $\mu$. Changing variables gives for $\mu \le 0$
    \begin{equation}\label{eq:250420261}
         \P\rk{\pn=\rhoc L^3}=\ex^{p(\mu)-p(0)-\beta\mu \rhoc L^3}\P_\mu\rk{\pn=\rhoc L^3}\, .
    \end{equation}
    By Lemma~\ref{lem:changeOfDens} and the monotonicity of $t\mapsto\ex^{\beta \mu j}$ in $\mu$, there exists $r=r_L\sim 8a_1\pi$ such that for $\mu=-r^2L^{-2}\log^2(L)$
    \begin{equation}
        \E_{\mu}\ek{\pn}=\rhoc L^3\, .
    \end{equation}
    This then yields by Lemma~\ref{lem:localCLT} and Lemma~\ref{lem:calculationOfExpectation} that for some $c>0$
    \begin{equation}\label{eq:25042026111}
        \frac{\sqrt{\log(L)}}{cL^2}\le \P_\mu\rk{\pn=\rhoc L^3}\le \frac{c\sqrt{\log(L)}}{L^2}\, .
    \end{equation}
    Thus, it is sufficient to analyze the term that describes the change of measure. Expand
    \begin{equation}
        p(\mu)-p(0)-\beta\mu\rhoc L^3=\sum_{j=1}^{\alpha L^2}\rk{\ex^{\beta \mu j}-1}\frac{t_j}{j}-\beta \mu a_0 (\beta j)^{-3/2}+\Ocal\rk{\mu L^3 (\alpha L^2)^{-1/2}}\, .
    \end{equation}
    The error term is $\Ocal\rk{\log^2(L)\alpha^{-1/2}}$. Set $\mu^+=L^{-2}/\log^{-1/2}(L)$. Expand
    \begin{multline}
        \sum_{j=1}^{1/\mu^+}\rk{\ex^{\beta \mu j}-1}\frac{t_j}{j}-\beta \mu a_0 (\beta j)^{-3/2}\\=\sum_{j=1}^{1/\mu^+}\rk{\ex^{\beta \mu j}-1}\frac{a_0(\beta j)^{-3/2}+a_1(\beta j)^{-1}(1+o(1))}{j}-\beta \mu a_0 (\beta j)^{-3/2}\, .
    \end{multline}
    The $a_0$ term is equal to
\begin{equation}
    a_0(-\beta{\mu})^{3/2}L^3\sum_{j=1}^{1/\mu^+}(-\beta\mu)\rk{\ex^{\beta \mu j}-1-\beta \mu j}(-\beta\mu j)^{-5/2}\, ,
\end{equation}
from which a Riemann-approximation yields
    \begin{equation}
        a_0r^3 \log^3(L)\int_0^\infty \rk{\ex^{-x}-1+x}x^{-5/2}\d x(1+o(1))=r^3\log^3(L)\frac{4\sqrt{\pi}}{3 (4\pi)^{3/2}}(1+o(1))\, .
    \end{equation}\
    The constant simplifies to $(6\pi)^{-1}$.
    
    The $a_1$ term is given by
\begin{equation}
    a_1L^2\beta \sum_{j=1}^{1/\mu^+}\rk{\ex^{\beta\mu j}-1}\frac{1}{(\beta j)^2}=a_1L^2\beta \sum_{j=1}^{1/\mu^+}\rk{\ex^{\beta\mu j}-1-\beta \mu j}\frac{1}{(\beta j)^2}+a_1L^2\beta \sum_{j=1}^{1/\mu^+}\frac{\mu }{(\beta j)}\, .
\end{equation}
Using the same Riemann-argument as before, the first term is of order $\Ocal\rk{\log^2(L)}$. The second term equals
\begin{equation}
    a_1L^2\beta \sum_{j=1}^{1/\mu^+}\frac{\mu }{(\beta j)}=-2a_1r^2\log^3(L)(1+o(1))\, .
\end{equation}
Hence, using the definition of $r$ yields
\begin{equation}
    p(\mu)-p(0)-\beta\mu\rhoc L^3=r^2\log^3(L)\rk{\frac{r}{6\pi}-2a_1}(1+o(1))=-128\pi^2a_1^3\log^3(L)/3(1+o(1))\, .
\end{equation}
Inserting this in \eqref{eq:250420261} yields with \eqref{eq:25042026111} the claim.
\end{proof}
It remains to disprove the existence of long loops. Using the previous change of measure argument, this can be reduced to bounding
\begin{equation}
    \P_\mu\rk{\pn\hk{cL^2/\log(L)}>0}\, .
\end{equation}
By the Poisson property, this equals to
\begin{equation}
   1- \exp\rk{-\sum_{j\ge cL^2/\log(L)}\frac{\ex^{\beta \mu j}}{j}t_j}\le \Ocal\rk{L^{-\tilde{c}}}\, ,
\end{equation}
where $\tilde{c}$ can be made arbitrarily large by adjusting $c$ (and therefore suppressing any polynomial factors from $t_j$). This is enough to show the claim.
\subsubsection{The one-particle reduced density matrix}\label{sec:oprdmposel}
It remains to analyze the one-particle reduced density matrix $\gamma$. Using \eqref{eq:250420261} yields
\begin{equation}
     \gamma_{L,\beta,\rhoc}(x,y)=\sum_{l=1}^{\rhoc L^d}\ex^{\beta \mu l}\frac{p_{\beta l}\hk{L}(x,y)\P_{\mu}\rk{\pn_L=\rhoc L^d-l}}{\P_{\mu}\rk{\pn_L=\rhoc L^d}}\, .
\end{equation}
By Lemma~\ref{lem:localCLT}
\begin{equation}
   c^{-1} \sum_{l=1}^{\nu^{-1}}\ex^{\beta \mu l}p_{\beta l}\hk{L}(x,y)\le  \sum_{l=1}^{\nu^{-1}}\ex^{\beta \mu l}\frac{p_{\beta l}\hk{L}(x,y)\P_{\mu}\rk{\pn_L=\rhoc L^d-l}}{\P_{\mu}\rk{\pn_L=\rhoc L^d}}\le c \sum_{l=1}^{\nu^{-1}}\ex^{\beta \mu l}p_{\beta l}\hk{L}(x,y)\, .
\end{equation}
Write now $A\asymp B$ if there exists $C>0$ with $C^{-1}A\le B\le C A$. By \cite[Theorem 1.1]{li2016heat}
\begin{equation}
   \sum_{l=1}^{\nu^{-1}}\ex^{\beta \mu l}p_{\beta l}\hk{L}(x,y)\asymp  \sum_{l=1}^{\nu^{-1}}l^{-3/2}\ex^{\beta \mu l}\ex^{-\abs{x-y}^2/(4l\beta)}\, .
\end{equation}
A Riemann approximation yields
\begin{equation}
    \sum_{l=1}^{\nu^{-1}}l^{-3/2}\ex^{\beta \mu l}\ex^{-\abs{x-y}^2/(4l\beta)}\asymp\int_0^\infty t^{-3/2}\ex^{\beta \mu t-\alpha^2L^2/(4t\beta)}\d t\, ,
\end{equation}
where $\abs{x-y}=L^2\alpha$. Denote $A=-\beta \log^2(L)r^2/L^2$ and $B=\alpha^2L^2/(4\beta)$. Then
\begin{equation}
    \int_0^\infty t^{-3/2}\ex^{\beta \mu t-\alpha^2L^2/(4t)}\d t=\rk{\frac{B}{A}}^{-1/4}\int_0^\infty t^{-3/2}\ex^{-\sqrt{AB}(t+1/t)}\d t\, .
\end{equation}
By the Laplace method, 
\begin{equation}
    \int_0^\infty t^{-3/2}\ex^{-\sqrt{AB}(t+1/t)}\d t= \ex^{-2\sqrt{AB}(1+o(1))}\, .
\end{equation}
This yields that
\begin{equation}
    \sum_{l=1}^{\nu^{-1}}l^{-3/2}\ex^{\beta \mu l}\ex^{-\abs{x-y}^2/(4l)}= L^{-1}\ex^{-2\sqrt{A B}(1+o(1))}=L^{-1-\alpha \pi a_1/2(1+o(1))}\, .
\end{equation}
It remains to bound
\begin{equation}
    \sum_{l=\nu^{-1}}^{\rhoc L^d}\ex^{\beta \mu l}\frac{p_{\beta l}\hk{L}(x,y)\P_{\mu}\rk{\pn_L=\rhoc L^d-l}}{\P_{\mu}\rk{\pn_L=\rhoc L^d}}\, .
\end{equation}
By Lemma~\ref{lem:localCLT}, the above can be bounded from above by
\begin{equation}
    C\sum_{l=\nu^{-1}}^{\rhoc L^d}\ex^{\beta \mu l}=\Ocal\rk{\ex^{-c\log^{3/2}(L)}}\, .
\end{equation}
The claim then follows.
\subsubsection{Positive first boundary heat coefficient, zero first eigenvalue}
The results remain unchanged. Here, $\mu^{-1}=L^2/\log^2(L)$ is the suppression scale, hence loops of length larger or equal to $L^2$ do not appear (see Lemma~\ref{lem:dowardsdeviations}). This essentially means that the sign of the first eigenvalue is not felt by the system. We refer the reader to \eqref{eq:260420261} for the weight of the long loops and remark that only Lemma~\ref{lem:changeOfDens} needs adaptation, by cutting loops off at $L^2$.

\subsection{Zero first boundary heat coefficient}
Note that in this case, $\lambda_1=0$, since the manifold has no boundary and hence no killing occurs. We keep this section slightly more general than just considering the case $\Lambda=\T^3$.

\begin{lemma}
    As $L\to\infty$\begin{equation}\label{eq:040520265}
    \E\ek{\pn\hk{\le L^2}}=\rhoc L^3+L^2c_\Delta(1+o(1))\, ,
\end{equation}
for some $c_\Delta\in \R$
\end{lemma}
\begin{proof}
    There are potentially two sources for the $L^2$ term: from the cut-off and from an $a_2$ term. We include the $a_2$ term, even though $a_2=0$ for the torus. However the calculation generalizes for other manifolds, so keep it general.

    Write $g(t)=Z(t)-a_0t^{-3/2}-a_2t^{-1/2}$. Then
    \begin{equation}
        a_0L^3\beta^{-3/2}\sum_{j=1}^{L^2}j^{-3/2}=\rhoc L^3-\beta^{-3/2}a_0(d/2-1)^{-1}L^2(1+o(1))\, .
    \end{equation}
    Furthermore
    \begin{equation}
         a_2L^1\beta^{-1/2}\sum_{j=1}^{L^2}j^{-1/2}=2\beta^{-1/2}a_2L^2(1+o(1))\, .
    \end{equation}
    Finally
\begin{equation}
    \sum_{j=1}^{L^2}g(\beta jL^{-2})=L^2+\sum_{j=1}^{L^2}\ek{g(\beta jL^{-2})-1}\, .
\end{equation}
Since $g(\beta jL^{-2})-1$ decays exponentially, the same Riemann approximation as in \eqref{eq:050520262} applies.
\end{proof}  

\eqref{eq:040520265} suggests that $\P_{L,\beta,\rhoc}\rk{\pn_L=\rhoc L^3}$ is a central limit type event and therefore should have a decay of order $L^{-2}$, matching the limit $a_1\to 0$ in the Lemma~\ref{lem:partdoubleneg}. However, an additional factor of $L^{-1}$ is present due to the suppression of the long loops.

The fact that short and long loops are on the same scale implies that the coarse bounds from Lemma~\ref{lem:upperboundd3negative} no longer suffices and we need to prove a local limit theorem.

\begin{lemma}[Local limit and left-tail bound]\label{lem:torusCLT}
Abbreviate $m_L=\E\ek{\pn\hk{\le L^2}}$ and for \(n\in\mathbb Z\), set
\begin{equation}
    x_{L,n}=\frac{n-m_L}{L^2}.
\end{equation}
Define
\begin{equation}
        \Psi(z)=\int_0^1
        \frac{Z(\beta u)}{u}
        (e^{zu}-1-zu)\d u,
    \qquad z\in i\R,
\end{equation}
and the associated density
\begin{equation}
        f_\Delta(x)=\frac{1}{2\pi}    \int_{\R}\exp\gk{-itx+\Psi(it)}
    \d t\, .
\end{equation}
Then \(f_\Delta\) is a continuous bounded density and, uniformly for
\(x_{L,n}\) in compact subsets of $\R$,
\begin{equation}\label{eq:torusCLTcompact}
    L^2\P\rk{\pn\hk{\le L^2}=n}=f_\Delta(x_{L,n})+o(1).
\end{equation}

Moreover, $M<\infty$ fixed, there are constants $C,c>0$
such that, uniformly in $n$ satisfying
\begin{equation}
        0\le \frac{m_L-n}{L^2}\le M\log^{1/3}L,
\end{equation}
one has
\begin{equation}\label{eq:torusCLTlefttail}
    \P\rk{\pn\hk{\le L^2}=n}\le\frac{C}{L^2}\rk{1+\frac{m_L-n}{L^2}}^{1/2}\exp\rk{-c\rk{\frac{m_L-n}{L^2}}^3}\}\, .
\end{equation}
In particular,
\begin{equation}\label{eq:torusCLTlefttailSimple}
    \P\rk{\pn\hk{\le L^2}=n} \le\frac{C}{L^2}\exp\rk{-c\rk{\frac{m_L-n}{L^2}}^3}\, .
\end{equation}
after changing \(C,c\).
\end{lemma}
\begin{proof}
Write
\begin{equation}
        u_{j,L}=\frac{j}{L^2}, \qquad
    q_{j,L}=\frac{t_j}{j}=\frac{1}{j}Z(\beta jL^{-2})\, ,
\end{equation}
so that
\begin{equation}
        \pn\hk{\le L^2}=\sum_{j\le L^2} jN_j,\qquad N_j\sim \Poi(q_{j,L})\, ,
\end{equation}
are independent. For
\begin{equation}
    X_L=\frac{\pn\hk{\le L^2}-m_L}{L^2}\, .
\end{equation}
we have
\begin{equation}\label{eq:PsiLdef}
    \log\E\ek{e^{zX_L}}=\Psi_L(z)=\sum_{j\le L^2}\frac{t_j}{j}\rk{\ex^{zu_{j,L}}-1-zu_{j,L}}\, .
\end{equation}

For $z$ in bounded subsets of $\{ \Re z\le 0\}$, Riemann approximation gives
\begin{equation}\label{eq:PsiLconv}
    \Psi_L(z)=\Psi(z)(1+o(1))=\int_0^1\frac{Z(\beta u)}{u}\rk{e^{zu}-1-zu}
    \d u\, .
\end{equation}
The integral is finite near $u=0$, since
\begin{equation}
        \ex^{zu}-1-zu=\Ocal(u^2),\qquad\frac{Z(\beta u)}{u}=\Ocal(u^{-5/2})\, .
\end{equation}
Fourier inversion gives for $x_{L,n}=(n-m_L)/L^2$,
\begin{equation}\label{eq:FourierInversionShort}
    L^2\P\rk{\pn\hk{\le L^2}=n}=\frac{1}{2\pi}\int_{-\pi L^2}^{\pi L^2}\exp\rk{-itx_{L,n}+\Psi_L(it)}\d t .
\end{equation}
The integrand is dominated by an integrable function. Indeed,
for $0< \abs{t}\le \eta L^2$, the leading heat-trace term and the inequality $1-\cos x\ge c x^2$ for $\abs{x}\le 1$ give
\begin{equation}
    \begin{split}
    -\Re \Psi_L(it)&=\sum_{j\le L^2}\frac{t_j}{j}\rk{1-\cos(tjL^{-2})}\\
    &\ge cL^3\sum_{j\le L^2/\abs{t}}j^{-5/2}\rk{\frac{\abs{t}j}{L^2}}^2\\
    &\ge c\abs{t}^{3/2}\, .
\end{split}
\end{equation}
For $\eta L^2\le \abs{t}\le \pi L^2$, the $j$ term already gives
\begin{equation}
    -\Re \Psi_L(it)\ge c_\eta L^3\, .
\end{equation}
Therefore
\begin{equation}
        \abs{\exp\rk{\Psi_L(it)}}\le C\exp\rk{-c\abs{t}^{3/2}}
\end{equation}
on the relevant Fourier scale, up to an exponentially small contribution
near $\abs{t}\asymp L^2$. Combining this domination with \eqref{eq:PsiLconv} proves the compact local limit \eqref{eq:torusCLTcompact}.

It remains to prove the left-tail estimate. Set
\begin{equation}
        K_L(\theta)=\log\E\ek{e^{-\theta X_L}}=\Psi_L(-\theta),
    \qquad \theta\ge 0 \, .
\end{equation}
For $1\le \theta\le C\log^{2/3}L$, the heat-trace expansion and the
change of variables \(v=\theta jL^{-2}\) give
\begin{equation}\label{eq:KLasy}
    K_L(\theta)
    =
    \kappa\,\theta^{3/2}
    +\Ocal(\theta),
\end{equation}
where
\begin{equation}\label{eq:kappadef}
\kappa=a_0\beta^{-3/2}\int_0^\infty\rk{\ex^{-v}-1+v}v^{-5/2}\d v=a_0\beta^{-3/2}\Gamma(-3/2)=\frac{1}{6\pi\beta^{3/2}}\, .
\end{equation}
Consequently,
\begin{equation}\label{eq:KLderivs}
    K_L'(\theta)=\frac{3}{2}\kappa\,\theta^{1/2}+\Ocal(1),\qquad  K_L''(\theta)=\frac{3}{4}\kappa\,\theta^{-1/2}(1+o(1))\, .
\end{equation}
Let
\begin{equation}
    x=\frac{m_L-n}{L^2}\, .
\end{equation}
For $1\le x\le M\log^{1/3}L$, choose $\theta=\theta_x$ such that
\begin{equation}
    K_L'(\theta_x)=x\, .
\end{equation}
By \eqref{eq:KLderivs},
\begin{equation}
    \theta_x\asymp x^2,\qquad K_L''(\theta_x)\asymp x^{-1}\, ,
\end{equation}
and
\begin{equation}
    \theta_xx-K_L(\theta_x)\ge c x^3\, .
\end{equation}
Tilt the measure by
\begin{equation}
        \frac{\d\P_{\theta_x}}{\d\P}=\exp\rk{-\theta_x X_L-K_L(\theta_x)}\, .
\end{equation}
Then
\begin{equation}
        \P\rk{\pn\hk{\le L^2}=n}=\exp\rk{K_L(\theta_x)-\theta_x x}\P_{\theta_x}\rk{\pn\hk{\le L^2}=n}\, .
\end{equation}
Under $\P_{\theta_x}$, $n$ is the tilted mean. The same
Fourier estimate as above, applied to the tilted characteristic function,
gives the local bound
\begin{equation}
     \sup_m \P_{\theta_x}\rk{\pn\hk{\le L^2}=m}\le \frac{C}{L^2\sqrt{K_L''(\theta_x)}}
\end{equation}
Using $K_L''(\theta_x)\asymp x^{-1}$, we obtain
\begin{equation}
        \P\rk{\pn\hk{\le L^2}=n}\le\frac{C(1+x)^{1/2}}{L^2}\exp\rk{-cx^3}.
\end{equation}
This proves \eqref{eq:torusCLTlefttail}.
\end{proof}
With the above lemma, the calculation of the partition function is straightforward
\begin{lemma}\label{lem:zeroboundaryPartition}
    There exists $C>0$ such that for all $L\ge 1$
    \begin{equation}
        C^{-1}L^{-3}\le \P_{L,\beta,\rhoc L^3}\rk{\pn_L=\rhoc L^3}\le CL^{-3}\, .
    \end{equation}
\end{lemma}
\begin{proof}
    Throughout the proof write $S_L=\pn\hk{\le L^2}$ and $R_L=\pn\hk{>L^2}$. The random variables $S_L$ and $R_L$ are independent. By \eqref{eq:040520265}, there exists $c_\L\in\R$ such that
    \begin{equation}\label{eq:zeroboundaryMeanShort}
        \E\ek{S_L}=\rhoc L^3+c_\L L^2+o\rk{L^2}\, .
    \end{equation}
    By independence,
    \begin{equation}\label{eq:zeroboundaryPartitionSplit}
        \P_{L,\beta,\rhoc L^3}\rk{\pn_L=\rhoc L^3}=\sum_{k=0}^{\rhoc L^3}\P\rk{R_L=k}\P\rk{S_L=\rhoc L^3-k}\, .
    \end{equation}
    We first prove the lower bound. Since $\lambda_1=0$, the first eigenfunction is constant and $t_j\ge 1$ for all $j$. Moreover, by \eqref{eq:firstEvDom}, $t_j\le C$ for $j\ge L^2$. Therefore
    \begin{equation}\label{eq:zeroboundaryNoLongLoops}
        \P\rk{R_L=0}=\ex^{-\sum_{L^2<j\le \rhoc L^3}t_j/j}\asymp L^{-1}\, .
    \end{equation}
    On the other hand, Lemma~\ref{lem:torusCLT} and \eqref{eq:zeroboundaryMeanShort} imply
    \begin{equation}\label{eq:zeroboundaryShortAtCritical}
        \P\rk{S_L=\rhoc L^3}\asymp L^{-2}\, .
    \end{equation}
    The term $k=0$ in \eqref{eq:zeroboundaryPartitionSplit} thus gives
    \begin{equation}
        \P_{L,\beta,\rhoc L^3}\rk{\pn_L=\rhoc L^3}\ge C^{-1}L^{-3}\, .
    \end{equation}
    It remains to prove the upper bound. Choose $M>0$ large enough and set $\up=M L^2\log^{1/3}(L)$. By Lemma~\ref{lem:dowardsdeviations} and \eqref{eq:zeroboundaryMeanShort},
    \begin{equation}\label{eq:zeroboundaryRestrictK}
        \sum_{k>\up}\P\rk{R_L=k}\P\rk{S_L=\rhoc L^3-k}\le \P\rk{S_L-\E\ek{S_L}\le -cM L^2\log^{1/3}(L)}=\Ocal\rk{L^{-4}}\, ,
    \end{equation}
    where $M$ was increased in the last step. The contribution of $k=0$ is bounded by $CL^{-3}$ by \eqref{eq:zeroboundaryNoLongLoops} and \eqref{eq:zeroboundaryShortAtCritical}. Moreover, $\P\rk{R_L=k}=0$ for $0<k\le L^2$. Hence it remains to consider $L^2<k\le\up$. In this regime,
    \begin{equation}\label{eq:zeroboundaryLongPointBound}
        \P\rk{R_L=k}=\P\rk{\pn\hk{L^2<\ell\le k}=k}\P\rk{\pn\hk{>k}=0}\le \frac{C}{L^2}\ex^{-\sum_{k<j\le \rhoc L^3}t_j/j}\le C\frac{k}{L^5}\, .
    \end{equation}
    Indeed, the first factor is bounded by Lemma~\ref{lem:distrMesco} with $\alpha=1$, while the second factor is bounded by $CkL^{-3}$ since $t_j\ge 1$. Finally, the non-compact part of Lemma~\ref{lem:torusCLT} gives
    \begin{equation}\label{eq:zeroboundaryShortTailBound}
        \P\rk{S_L=\rhoc L^3-k}\le \frac{C}{L^2}\ex^{-c\rk{kL^{-2}-C}_+^3}\, .
    \end{equation}
    Combining \eqref{eq:zeroboundaryLongPointBound} and \eqref{eq:zeroboundaryShortTailBound}, we obtain
    \begin{equation}
        \sum_{L^2<k\le\up}\P\rk{R_L=k}\P\rk{S_L=\rhoc L^3-k}\le \frac{C}{L^7}\sum_{k>L^2}k\ex^{-c\rk{kL^{-2}-C}_+^3}\le CL^{-3}\, .
    \end{equation}
    Together with \eqref{eq:zeroboundaryRestrictK} and the bounds on the $k=0$ term, this proves the upper bound.
\end{proof}
\begin{lemma}
    The one-particle reduced density matrix $\gamma_{L,\beta,\rhoc}(x,y)$ satisfies for $x,y$ with $\dist(x,y)>\e L$, $\dist(x,\partial \L_L)>\e L$ and $\dist(y,\partial \L_L)>\e L$
    \begin{equation}
    \gamma_{L,\beta,\rhoc}(x,y)\asymp L^{-1}\, .
    \end{equation}
\end{lemma}
\begin{proof}
    Being similar to previous computations, we sketch only the main steps.

    On account of Lemma~\ref{lem:torusCLT}
    \begin{equation}
        \gamma_{L,\beta,\rhoc}(x,y)\asymp \sum_{r=1}^{L^2}\frac{p_{r\beta }\hk{L}(x,y)\P_{L,\beta,\rhoc}\rk{\pn_L=\rhoc L^3-r}}{\P_{L,\beta,\rhoc}\rk{\pn_L=\rhoc L^3}}\asymp \sum_{r=1}^{L^2}p_{r\beta }\hk{L}(x,y)\, .
    \end{equation}
    Using \cite[Theorem 6.3.8]{a1981operator} for $r\le L^2$ and \eqref{eq:spectraldec} for $r\ge L^2$ (together with the fact that $\phi\hk{1}\equiv 1$) yields
    \begin{equation}
         \gamma_{L,\beta,\rhoc}(x,y)\asymp \sum_{r=1}^{L^2}r^{-3/2}\ex^{-c L^2/r}\asymp L^{-1}\, .
    \end{equation}
\end{proof}
\section{Proof in higher dimensions}
The proof in higher dimension is significantly easier, as the second-order term is larger than the central limit theorem fluctuations.
\begin{lemma}\label{lem:highdcalculationOfExpectation}
    The expectation of $\ps$ satisfies
    \begin{multline}
        \E\ek{\ps}=\rhoc L^d+\frac{a_1 L^{d-1}}{\beta^{(d-1)/2}}\zeta(d/2-1/2)\\
        +\frac{a_2L^{d-2}}{\beta^{(d-2)/2}}\rk{2\log(L)\1\gk{d=4}+\zeta(d/2-1)\1\gk{d\ge 5}}+\Ocal\rk{L^{d-2}}\, .
    \end{multline}
\end{lemma}
\begin{proof}
    The proof is similar to that of Lemma~\ref{lem:calculationOfExpectation}, but an additional term needs to be included: set $g(t)=Z(t)-a_0t^{-d/2}-a_1t^{-(d-1)/2}-a_2t^{-(d-2)/2}$. The $a_0$ term yields
    \begin{equation}
a_0\beta^{-d/2}L^d\sum_{j=1}^{L^2}j^{-d/2}\sim a_0\beta^{-d/2}L^d\rk{\zeta(d/2)-\rk{d/2-1}^{-1}L^{2-d}}\, .
    \end{equation}
    The $a_1$ term gives
    \begin{equation}
        a_1\beta^{-(d-1)/2}L^{d-1}\sum_{j=1}^{L^2}j^{(1-d)/2}\sim  a_1\beta^{-(d-1)/2}L^{d-1}\rk{\zeta(d/2-1/2)-\rk{(d-1)/2-1}^{-1}L^{3-d}}\, .
    \end{equation}
    Finally, the $a_2$ term equals
    \begin{equation}
        a_2\beta^{-(d-2)/2}L^{d-2}\sum_{j=1}^{L^2}j^{(2-d)/2}\sim a_2\beta^{-(d-2)/2}L^{d-2}\rk{2\log(L)\1\gk{d=4}+\zeta(d/2-1)\1\gk{d\ge 5}}+\Ocal(L^2)\, .
    \end{equation}
    Finally, similar to before,
    \begin{equation}
        L^d\sum_{j=1}^{L^2}g(\beta L^{-2}j)=\Ocal\rk{L^{d-2}}\, .
    \end{equation}
\end{proof}
\subsection{Negative boundary elasticity}
Next, an analogue of Proposition~\ref{prop:clt} is given. The limit law now changes from Fredholm to Gaussian.
\begin{proposition}\label{prop:highdCLT}
    Set $b_L=L^{d/2}\log(L)^{\1\gk{d=4}/2}$. Then, $\pn_L$ satisfies a Gaussian central limit theorem with mean zero and variance given by
    \begin{equation}
        \frac{2\1\gk{d=4}+\zeta(d/2-1)\1\gk{d\ge 5}}{(4\pi\beta)^{d/2}}\, .
    \end{equation}
\end{proposition}
\begin{proof}
    Again, expand
    \begin{equation}
         \E\ek{\exp\rk{is\rk{\pn-\E[\pn]}}}=\exp\rk{\sum_{j=1}^{\rho L^d}\frac{1}{j}\rk{\ex^{isj}-1-isj}t_j}\, .
    \end{equation}
    For $t\in\R$ fixed and $s=t/b_L$, one obtains
    \begin{equation}
        \sum_{j=1}^{b_L/\log(L)}\frac{1}{j}\rk{\ex^{isj}-1-isj}t_j=-\frac{t^2L^d}{2b_L^2}\beta^{-d/2}\sum_{j=1}^{b_L/\log(L)}a_0 j^{1-d/2}(1+o(1))=-\frac{c_dt^2}{2(4\pi\beta)^{d/2}}(1+o(1))\, ,
    \end{equation}
    where 
    \begin{equation}
        c_d=\begin{cases}
            2&\textnormal{ if }d=4\, ,\\
            \zeta(d/2-1)&\textnormal{ if }d\ge 5\, .
        \end{cases}
    \end{equation}
    When $d\ge 5$, one finds that
    \begin{equation}
         \sum_{j\ge b_L/\log(L)}\frac{1}{j}\abs{\ex^{isj}-1-isj}t_j=o(1)\, ,
    \end{equation}
    as $L^{d/2}/\log(L)\gg L^2$ and hence it is suppressed by the exponential decay from \eqref{eq:firstEvDom}.

    For $d=4$,
    \begin{equation}
         \sum_{j= b_L/\log(L)}^{L^2}\frac{1}{j}\abs{\ex^{isj}-1-isj}t_j\le C\log(L)^{-1}\sum_{j= b_L/\log(L)}^{L^2}j^{-1}=o(1)\, .
    \end{equation}
    The sum over $j\ge L^2$ can again be ignored by the exponential decay from \eqref{eq:firstEvDom}.
\end{proof}
\begin{lemma}\label{lem:highdPartfun}
    For $a_1<0$, the partition function is given by
    \begin{equation}
        \P\rk{\pn_L=\rhoc L^d}=\exp\rk{a_1\lambda_1 \beta^{(3-d)/2}\zeta(d/2-1/2)L^{d-3}(1+o(1))}\, .
    \end{equation}
\end{lemma}
\begin{proof}
    Lemma~\ref{lem:upperboundd3negative} continues to hold without major modifications. Hence, the mass of particles in long loops equals $-a_1L^{d-1}\beta^{(1-d)/2}$ up to first order. Adapting Lemma~\ref{lem:distrMesco} yields that
    \begin{equation}
        \P\rk{\pmes=-a_1\beta^{(1-d)/2}L^{d-1}}=\exp\rk{a_1\lambda_1 \beta^{(3-d)/2}\zeta(d/2-1/2)L^{d-3}(1+o(1))}\, .
    \end{equation}
    Lower order contributions are subsumed in the $o(1)$ term. From there on, one can imitate the proof of the three-dimensional case.
\end{proof}
Since the CLT-scale and the expectation scale from Lemma~\ref{lem:highdcalculationOfExpectation} are well-separated, the concentration is stronger.
\begin{lemma}\label{lem:upperbounddhighnegative}
    As $L\to\infty$, for every $r>0$ with $r\le \log\log\log(L)$
    \begin{equation}
        \P\rk{\abs{\ps-\E[\ps]}\ge \kappa L^{d-1}}=\Ocal\rk{\ex^{-r\kappa L^{d-3}}}\, .
    \end{equation}
\end{lemma}
\begin{proof}
    The proof is similar to that of Lemma~\ref{lem:upperboundd3negative}. Choose $t=\beta \lambda_1 L^{-2}+rL^{-2}$. For $d\ge 4$, we have
    \begin{equation}
        \sum_{j=1}^{L^2/r}\frac{1}{j}\rk{\ex^{t j}-1-tj}t_j\le C r^2L^{d-4}\log(L)^{\1\gk{d=4}}\, .
    \end{equation}
    On the other hand $t\kappa L^{d-1}\ge r \kappa L^{d-3}$ and hence the short loops are suppressed as for $d=3$. The estimate for the loops with duration exceeding $L^2/r$ is analogous to the $d=3$ case and is omitted.
\end{proof}
The computation of the one-particle reduced density matrix is now as before: for $\e>0$, by Lemma~\ref{lem:upperbounddhighnegative} and Lemma~\ref{lem:highdPartfun}, it follows that
\begin{equation}
    \gamma_{L,\beta,\rhoc}(x,y)\sim \sum_{r=(-a_1\beta^{(1-d)/2}-\e)L^{d-1}}^{(-a_1\beta^{(1-d)/2}+\e)L^{d-1}}\frac{p_{\beta r}\hk{L}(x,y)\P_{L,\beta,\rhoc}\rk{\pn_L=\rhoc L^d-r,\ps\in D_+}}{\P_{L,\beta,\rhoc}\rk{\pn_L=\rhoc L^d}}\, .
\end{equation}
Since the loop are well mixed, Lemma~\ref{lem:distrMesco} applies (the proof does not need to be modified as the long loops do not ``see" the dimension of the space) and hence
\begin{equation}
     \gamma_{L,\beta,\rhoc}(x,y)\sim -\frac{a_1\zeta(d/2-1/2)}{\beta^{(d-1)/2}L}\phi\hk{1}(x/L)\phi\hk{1}(y/L)\, .
\end{equation}
We leave the details to the interested reader.
\subsection{Positive boundary elasticity}
For $a_1>0$, the picture changes drastically.
\begin{lemma}\label{lem:highdimchangedens}
    For $r>0$, choose
    \begin{equation}\label{eq:muchosenhighdim}
        \mu=-\frac{r}{L\log(L)^{\1\gk{d=4}}}\, .
    \end{equation}
    Then, as $L\to\infty$ for $a_1>0$
    \begin{equation}
        \rho(\mu)-\rho(0)\sim-\beta^{1-d/2}(4\pi)^{-d/2} rL^{d-1}\times\begin{cases}
            1&\textnormal{ if }d=4,\\
            \zeta(d/2-1)&\textnormal{ if }d\ge 5\, .
        \end{cases}\, .
    \end{equation}
\end{lemma}
\begin{proof}
    We give the main steps of the proof and refer to the proof of Lemma~\ref{lem:changeOfDens} for more details.
    
    Abbreviate
    \begin{equation}
        c_d=\1\gk{d=4}+\zeta(d/2-1)\1\gk{d\ge 5}\, .
    \end{equation}
    
    The main contribution comes from
    \begin{equation}
        \sum_{j=1}^{1/o(\mu)}\rk{\ex^{\beta \mu j}-1}t_j\sim \beta\mu L^d(4\pi \beta)^{-d/2}\sum_{j=1}^{1/o(\mu)}j^{1-d/2}\sim - \beta^{1-d/2}(4\pi)^{-d/2} rL^{d-1}c_d(1+o(1))\, ,
    \end{equation}
    where for $d=4$, the logarithm appears from the sum of the $j^{1-d/2}$.
\end{proof}
The above estimate gives the decay of the partition function in the regime $a_1>0$, as it for $d=3$. For this, the first step is a local central limit theorem.
\begin{lemma}
    Choose
    \begin{equation}
        \nu=\frac{1}{L^{d/2}\rk{1+\1\gk{d=4}\sqrt{\log(L)}}}\, .
    \end{equation}
    Then under $\P_\mu$ with $\mu$ chosen as in \eqref{eq:muchosenhighdim}, $\pn_L$ satisfies a local central limit theorem with Gaussian limit, with zero mean and inverse-variance given by $(4\pi \beta)^{-2}$ for $d=4$ and $(4\pi\beta)^{-d/2}/\zeta(d/2-1)$ for $d\ge 5$.
\end{lemma}
Note that $\nu$ is the same for $d=3$ and $d=4$. However, the particle number is $L^d$ and hence the three-dimensional case exhibits much wider fluctuations. 
\begin{proof}
    We sketch the main asymptotics of the contribution; the control of the errors is analogous to the proof of Lemma~\ref{lem:localCLT}.

    For $\e>0$ small enough, the main contribution comes from
    \begin{equation}
        \sum_{j=1}^{\mu^{-1-\e}}\frac{\ex^{\beta \mu j}}{j}\rk{\ex^{itj\nu}-1-itj\nu}t_j\sim -\frac{t^2\nu^2}{2}L^d (4\pi\beta)^{-d/2}\sum_{j=1}^{\mu^{-1-\e}}\frac{\ex^{\beta \mu j}}{j}j^{-d/2}\, .
    \end{equation}
    The asymptotics of the polylogarithm imply that
    \begin{equation}
        \sum_{j=1}^{\mu^{-1-\e}}\ex^{\beta \mu j}j^{1-d/2}\sim\begin{cases}
            \log(L)&\textnormal{ if }d=4,\\
            \zeta(d/2-1)&\textnormal{ if }d\ge 5\, .
        \end{cases}
    \end{equation}
    This concludes the proof.
\end{proof}
\begin{lemma}
    For $a_1>0$, as $L\to\infty$ for $d\ge 5$
    \begin{equation}
        \P\rk{\pn_L=\rhoc L^d}=\exp\rk{ -a_1^2 (4\pi)^{d/2}\beta^{1-d/2}\frac{\zeta(d/2-1/2)-1}{\zeta(d/2-1)}L^{d-2}\rk{1+o(1)}}\, .
    \end{equation}
    For $d=4$, the same limit gives
    \begin{equation}
        \P\rk{\pn_L=\rhoc L^4}=\exp\rk{ -a_1^2 (4\pi)^{2}\beta^{-1}\rk{\zeta(3/2)-1}\frac{L^{2}}{\log(L)}\rk{1+o(1)}}\, .
    \end{equation}
\end{lemma}
\begin{proof}
    Fix $d\ge 5$ first. Lemma~\ref{lem:highdcalculationOfExpectation} and Lemma~\ref{lem:highdimchangedens} imply that for some sequence $r_L=r$ satisfying
    \begin{equation}\label{eq:choiceofrhighd}
        r\sim \frac{a_1(4\pi)^{d/2}}{\sqrt{\beta}\zeta(d/2-1)}\, ,
    \end{equation}
    that
    \begin{equation}
        \E_\mu\ek{\pn_L}=\rhoc L^{d}\, .
    \end{equation}
    Using the same strategy as in the proof of Proposition~\ref{prop:posboundarypart}, it remains to calculate
    \begin{equation}
        \sum_{j=1}^{\rhoc L^d}\rk{\frac{\ex^{\beta j\mu }-1}{j}t_j-(4\pi\beta)^{-d/2}\beta \mu j^{-d/2}}\, .
    \end{equation}
    Expanding $t_j=L^d(4\pi j\beta)^{-d/2}+L^{d-1}a_1(\beta j)^{-d/2+1/2}(1+o(1))$ in the contributing region of the $j$'s, one obtains for the first term
    \begin{equation}
        (4\pi\beta)^{-d/2}L^d\beta^2\mu^2L^d\sum_{j=1}^{\infty}j^{1-d/2}\sim{a_1^2 (4\pi)^{d/2}\beta^{-d/2+1}} \zeta(d/2-1)^{-1}L^{d-2}\, .
    \end{equation}
    The second term gives
    \begin{equation}
        \beta \mu L^{d-1}a_1\beta^{(1-d)/2}\sum_{j=1}^{\infty}j^{1/2-d/2}\sim - a_1^2(4\pi)^{d/2}\beta^{(2-d)/2}\frac{\zeta(d/2-1/2)}{\zeta(d/2-1)}L^{d-2}\, .
    \end{equation}
    This concludes the proof for the case $d\ge 5$.

    For $d=4$, choosing $r=a_1(4\pi)^2\beta^{-1/2}$ gives an $a_0$ contribution of
    \begin{equation}
        (4\pi)^{-2}\beta^2 L^4\beta^2 r^2 L^2\log^{-2}(L)\sum_{j=1}^{1/\mu}j^{-1}\sim a_1^2(4\pi)^2\beta^{-1}\frac{L^2}{\log(L)}\, .
    \end{equation}
    The $a_1$-contribution equals
    \begin{equation}
         \beta \mu L^{3}a_1\beta^{-3/2}\sum_{j=1}^{\infty}j^{1/2-2}\sim -a_1^2(4\pi)^2\beta^{-1}\zeta(3/2)\frac{L^2}{\log(L)}(1+o(1))\, .
    \end{equation}
    This concludes the proof in the case $d=4$.
\end{proof}
It remains to calculate the one-particle reduced density-matrix. The proof is similar to that of Section~\ref{sec:oprdmposel}, and hence we omit most details.

First, by Lemma~\ref{lem:highdcalculationOfExpectation} and Lemma~\ref{lem:highdimchangedens}, choosing $r$ as in \eqref{eq:choiceofrhighd} yields the correct expectation.

The main contribution is again due to
\begin{equation}
    \int_0^\infty t^{-d/2}\ex^{\beta \mu t-\alpha L^2/(4t)}\d t\asymp \ex^{-2\sqrt{-\beta \mu L^2/4}}=\exp\rk{-\frac{\sqrt{ar\beta}}{2}}\, .
\end{equation}
For $d\ge 5$, this gives that
\begin{equation}\label{eq:040520263}
    \gamma_{L,\beta,\rhoc}(x,y)=\exp\rk{-\sqrt{\frac{a_1\alpha^2(4\pi)^{d/2}}{4\zeta(d/2-1)}L}(1+o(1))}\, .
\end{equation}
For $d=4$, the same calculation yields a slightly slower decay
\begin{equation}\label{eq:040520264}
    \gamma_{L,\beta,\rhoc}(x,y)=\exp\rk{-\sqrt{\frac{a_1\alpha^2(4\pi)^{2}}{4\log(L)}L}(1+o(1))}\, .
\end{equation}
\subsection{Torus in high dimensions}
For the torus in $d\ge 4$, we first note that $a_i=0$ for $i\ge 1$, see Section~\ref{sec:geoandBC}. Next, choose $\alpha=\alpha_L=\log\log(L)$.
\begin{lemma}\label{lem:cltTorusHighD}
    $\pn\hk{\le \alpha L^2}$ satisfies a Gaussian central limit theorem with scale
    \begin{equation}\label{CLT:scale}
        \nu^{-1}=\begin{cases}
            L^2\sqrt{\log L}&\textnormal{ if }d=4\, ,\\
            L^{d/2}&\textnormal{ if }d\ge 5\, .
        \end{cases}
    \end{equation}
\end{lemma}
\begin{proof}
    This is similar to the previous central limit theorem proof of Lemma~\ref{lem:torusCLT}. However, it is easier, as the limiting law is Gaussian. Note that the variance of $\pn\hk{\le \alpha L^2}$ is given by
    \begin{equation}
        \sum_{j=1}^{\alpha L^2}j t_j\sim \begin{cases}
            2a_0L^4\log(L)&\textnormal{ if }d=4\, ,\\
            a_0\zeta(d/2-1)L^d&\textnormal{ if }d\ge 5\, .
        \end{cases}
    \end{equation}
    Note that this is on the order $\nu^{-2}$. Analogously to Lemma~\ref{lem:localCLT}, we find that for $t$ in bounded sets
    \begin{equation}
        \log\E\ek{\exp\rk{i{t}{\nu}\rk{\pn\hk{\le \alpha L^2}-\E[\pn\hk{\le \alpha L^2}]}}}=-\frac{t^2}{2c_d}(1+o(1))\, ,
    \end{equation}
    as $L\to\infty$.
\end{proof}
\begin{lemma}
    There exists $C>0$
    \begin{equation}
        C^{-1}L^{-d}\le \P_{L,\beta,\rhoc}\rk{\pn_L=\rhoc L^d}\le CL^{-d}\, .
    \end{equation}
    Furthermore, as $L\to\infty$
    \begin{equation}
    \gamma_{L,\beta,\rhoc}(x,y)\asymp 
    \begin{cases}
        L^{-2}\sqrt{\log(L)}, & d=4,\\
        L^{-d/2}, & d\ge5.
    \end{cases}
    \end{equation}
\end{lemma}
\begin{proof}
    \textbf{Partition function}: we expand
    \begin{equation}
        \P\rk{\pn_L=\rhoc L^d}=\sum_{k=0}^{\rhoc L^d}\P\rk{\pn\hk{\le \alpha L^2}=k}\P\rk{\pn\hk{> \alpha L^2}=\rhoc L^d-k}\, .
    \end{equation}
    Recall that in the expansion of Lemma~\ref{lem:highdcalculationOfExpectation}, $a_2=0$ for $\Lambda=\T^d$. Choose now $D_+$ the set of likely particle densities of short loops
    \begin{equation}
        D_+=\ek{\rhoc L^d-\alpha \nu^{-1},\rhoc L^d}\, ,
    \end{equation}
    enveloping $\E\ek{\pn\hk{\le \alpha L^2}}$ beyond the CLT scale, see Lemma~\ref{lem:cltTorusHighD}, i.e., $\P\rk{\pn\hk{\le \alpha L^2}\in D_+}=1/2+o(1)$. By the concentration result from Lemma~\ref{lem:upperboundd3negative}
    \begin{equation}
        \P\rk{\pn_L=\rhoc L^d}\sim \sum_{k\in D_+}\P\rk{\pn\hk{\le \alpha L^2}=k}\P\rk{\pn\hk{> \alpha L^2}=\rhoc L^d-k}\, .
    \end{equation}
    Note that Lemma~\ref{lem:distMetro0} does not carry any dimensional dependence. Hence, for $0\le l\le \alpha \nu$, we find that
    \begin{equation}
        \P\rk{\pn\hk{> \alpha L^2}=l}=\P\rk{\pn\hk{\alpha L^2<\ell\le \alpha \nu^{-1}}=l}\P\rk{\pn\hk{> \alpha \nu^{-1}}=0}\sim \frac{p_1(1)}{\alpha \nu^{-1}}\frac{\alpha \nu^{-1}}{\rhoc L^d}=\frac{p_1(1)}{\rhoc L^d}\, .
    \end{equation}
    Combining this with Lemma~\ref{lem:cltTorusHighD} yields
    \begin{equation}
        \P\rk{\pn_L=\rhoc L^d}\sim \frac{1}{2}\frac{p_1(1)}{\rhoc L^d}\, ,
    \end{equation}
    which concludes the asymptotics of the partition function.

    \noindent
    \textbf{One-particle reduced density-matrix}: we employ the same strategy as before: relevant contributions come from open paths of the same length as the central limit theorem scale. Indeed, by Lemma~\ref{lem:cltTorusHighD}
    \begin{equation}
        \gamma_{L,\beta,\rhoc}(x,y)\asymp \sum_{r=1}^{\nu^{-1}}\frac{p_{r\beta }\hk{L}(x,y)\P_{L,\beta,\rhoc}\rk{\pn_L=\rhoc L^d-r}}{\P_{L,\beta,\rhoc}\rk{\pn_L=\rhoc L^d}}\asymp \sum_{r=1}^{\nu^{-1}}p_{r\beta }\hk{L}(x,y)\, .
    \end{equation}
    The small paths are now negligible
    \begin{equation}
        \sum_{r=1}^{L^2}p_{r\beta }\hk{L}(x,y)\le C\sum_{r=1}^{L^2}r^{-d/2}\ex^{-\abs{x-y}^2/L^2}=\Ocal\rk{L^{2-d}}\, ,
    \end{equation}
    for $\abs{x-y}>\e L^2$ and $\e>0$ fixed. On the other hand, for $r\ge L^2$, mixing gives \begin{equation}
        p_{r\beta }\hk{L}(x,y)\asymp L^{-d}\, .
    \end{equation}
    Hence
    \begin{equation}
        \sum_{r=L^2}^{\nu^{-1}}p_{r\beta }\hk{L}(x,y)\asymp \frac{\nu^{-1}}{L^2 L^d}=\begin{cases}
            L^{-2}\sqrt{\log (L)}&\textnormal{ if }d=4\, ,\\
            L^{-d/2}&\textnormal{ if }d\ge 5\, ,
        \end{cases}
    \end{equation}
    on account of \eqref{CLT:scale}. This concludes the proof.
\end{proof}
\section*{Acknowledgments}
The author would like to thank Alexis Prévost for suggesting to look at the critical case.
\section*{Declaration on the Use of Generative AI}

Generative AI tools were used for language editing and to improve the clarity and readability of the manuscript. They were not relied upon to establish any mathematical results or proofs. The author has
independently verified all mathematical statements and takes full responsibility for the content of the manuscript.
\bibliographystyle{alpha}
\bibliography{thoughts.bib}
\end{document}